\documentclass[fontsize=12pt,a4paper,headings=normal,
twoside=false,leqno,parskip=half-,abstract=true]{scrartcl}

\usepackage[notref,notcite,color,final
]{showkeys}

\usepackage[english]{babel}
\usepackage[utf8]{inputenc}
\usepackage{hyperref}
\hypersetup{
 bookmarks=true,
 pdftitle={Sensitivity of chemical reaction networks: a structural approach 3.~Regular multimolecular systems},
 pdfauthor={Bernhard Brehm, Bernold Fiedler},
 colorlinks=true,
 linkcolor=blue,
 citecolor=blue,
 filecolor=blue,
 urlcolor=blue}

\usepackage{graphicx}
\usepackage[format=plain,labelfont=bf]{caption}
\usepackage{subfigure}
\usepackage{color}

\usepackage{amsmath,amsthm}
\usepackage{amssymb} 
\usepackage{latexsym}

\definecolor{refkey}{rgb}{1,0,0}
\definecolor{labelkey}{rgb}{1,0,0}

\usepackage{verbatim}
\usepackage{placeins}
\usepackage{array}
\usepackage[table]{xcolor}
\usepackage{hhline}

\newcommand{\cbox}[1]{\raisebox{\depth}{\fcolorbox{black}{#1}{\null}}}
\let\depth\relax

\usepackage[textwidth=2cm]{todonotes}

  \mathchardef\ordinarycolon\mathcode`\:
  \mathcode`\:=\string"8000
  \begingroup \catcode`\:=\active
    \gdef:{\mathrel{\mathop\ordinarycolon}}
  \endgroup

\theoremstyle{plain}
\newtheorem{thm}{Theorem}[section]
\newtheorem{lem}[thm]{Lemma}

\begin{document}

\title{{\LARGE{Sensitivity of chemical reaction networks:\\a structural approach.\\3.~Regular multimolecular systems}}
	}

\author{
Bernhard Brehm*\\
Bernold Fiedler*\\
\vspace{2cm}}

\date{version of \today}
\maketitle
\thispagestyle{empty}

\vfill

*\\
Institut für Mathematik\\
Freie Universität Berlin\\
Arnimallee 7\\ 
14195 Berlin, Germany


\newpage
\pagestyle{plain}
\pagenumbering{roman}
\setcounter{page}{1}

\begin{abstract}
We present a systematic mathematical analysis of the qualitative steady-state response to rate perturbations in large classes of reaction networks.
This includes multimolecular reactions and allows for catalysis, enzymatic reactions, multiple reaction products, nonmonotone rate functions, and non-closed autonomous systems.
Our \emph{structural sensitivity analysis} is based on the stoichiometry of the reaction network, only.
It does not require numerical data on reaction rates.
Instead, we impose mild and generic nondegeneracy conditions of algebraic type.
From the structural data, only, we derive which steady-state concentrations are sensitive to, and hence influenced by, changes of any particular reaction rate -- and which are not.
We also establish transitivity properties for influences involving rate perturbations.
This allows us to derive an \emph{influence graph} which globally summarizes the influence pattern of any given network.
The influence graph allows the computational, but meaningful, automatic identification of functional subunits in general networks, which hierarchically influence each other.
We illustrate our results for several variants of the glycolytic citric acid cycle.
Biological applications include enzyme knockout experiments, and metabolic control.
\end{abstract}

\tableofcontents


\newpage
\pagenumbering{arabic}
\setcounter{page}{1}

\section{Introduction}
\label{sec1:intro}

\numberwithin{equation}{section}
\numberwithin{figure}{section}

For large classes of biological, chemical or metabolic reaction networks, detailed numerical data on reaction rates are neither available, nor accessible, by parameter identification.
See the large, and growing, data bases on chemical and metabolic pathways like \cite{BernhardD, KEGG}, for thousands of examples.
One standard approach to establish, and check, the validity of such networks are \emph{knockout experiments}:
some reaction is obstructed, via the knockout of its catalyzing enzyme, and the response of the network is measured, e.g., in terms of concentration changes of metabolites.
The large area of \emph{metabolic control} studies how reaction rates steer the network to desired behaviour, or switch between different tasks; see for example \cite{HeinrichSchuster, fell92, stephanopoulos} and the references there.

\par
In this setting it is our goal to develop a reliable, and largely automatic, mathematical tool to aid our systematic understanding of large networks. Chemical reaction networks consist of reactions among metabolites. We consider reactions, quite generally, to cover biological and chemical reactions. Since we do not incorporate temperature dependence, explicitly, our setting is isothermal. The reacting chemical species, often called metabolites in biological settings, are denoted by labels $m$. 

More specifically, we address the response of steady states to rate perturbations in the network.
Here steady state refers to any time-independent long-term state of the system.
``Long-term'' refers to the relevant time scales of the model.
Time-periodic or chaotic responses are excluded, at present.

In experiments, only those steady states may be observable which are stable, or at least metastable on the relevant time scale.
Our mathematical approach is not limited by any stability or hyperbolicity requirements other than some mild nondegeneracy assumption.
For simplicity we present our approach in a local setting of linearized steady state response to small perturbations. In the concluding discussion we indicate how our results extend to the global setting of knockout experiments.

We derive a qualitative sensitivity matrix for the steady state response, which we encode as a \emph{flux influence relation} $j^*$ influences $j'$; in symbols: $j^* \leadsto j'$. Here $j^*$ indicates the label of the perturbed reaction rate, given by the experimental setting, and $j^* \leadsto j'$ indicates a \emph{nonzero resulting flux change} of the reaction with label $j'$ in the network. 
The flux of $j'$ measures the rate of conversion from input metabolites of reaction $j'$ to output metabolites. The flux change measures the change of that conversion rate, at steady state, under the influence of the external rate perturbation of reaction $j^*$. Note how the roles of $j^*$ and $j'$ in the influence relation $j^* \leadsto j'$ are subtly different. For example, $j^*$ may, or may not, influence $j':=j^*$ itself.

Although flux changes are the more convenient object, mathematically, they are less accessible in experiments. We therefore include any \emph{nonzero resulting concentration change} of any metabolite $m'$, at steady state, into our analysis; in symbols: $j^* \leadsto m'$. To combine both the flux and the concentration aspect, notationally, we define the \emph{influence relation}
	\begin{equation}
	j^* \leadsto \beta\,,
	\label{eq:1.18}
	\end{equation}
to indicate that the effect of the perturbation $j^*$ on $\beta$ is a nonzero change of $\beta$. Here $\beta$ denotes either a reaction $j'$ or a metabolite $m'$. In this case of a \emph{nonzero influence} $j^* \leadsto \beta$ we also say that the reaction or metabolite $\beta$ is \emph{sensitive to} $j^*$.

The influence relation \eqref{eq:1.18}, in itself, does not carry much information beyond a casuistic tabulation of who-influences-who. It is a transitivity property, which makes flux influence a central tool for the understanding of sensitivity results in metabolic networks.  
Consider any 2-step chain of influences $j^* \leadsto j \leadsto j'$, where $j^*, j, j'$ denote reaction labels in the network. As a consequence, direct influence $j^* \leadsto j'$ will be established. Only this \emph{transitivity of influence}, 
	\begin{equation}
	j^* \leadsto j \leadsto j' \quad \mathrm{implies}\quad 	j^* \leadsto j',
	\label{eq:1.18trans}
	\end{equation}
justifies the notion of a \emph{hierarchy of influence}. 

It is the notion of nonzero influence, together with this central transitivity property of the flux influence relation, which will identify meaningful units in a network and will allow us to sum up all results of rate perturbation experiments in a single \emph{flux influence graph} $\mathcal{F}$ below.

One first phenomenological indication for the mathematical structure of flux influence, which motivated our detailed study, was the apparent \emph{sparsity of sensitivity}:
Given a rate perturbation of a specific reaction $j^*$,
many $\beta$ are not influenced at all, because the reaction flux or metabolite concentration associated to 
$\beta$ does not change.
Such \emph{zero response} is the counterpart of the flux influence graph $\mathcal{F}$.
In fact, any zero response is a rational and mathematically rigorous test to any purported pathway structure:
\emph{Any experimentally validated nonzero response, above error threshold, which contradicts a zero entry in the sensitivity matrix of influences, falsifies the underlying pathway.}

We briefly comment on previous work in this series of papers, even though the present paper is self-contained and does not require any but the most standard mathematical prerequisites. The present paper is a sequel to \cite{fiedler15example, fiedler15mono}. In \cite{fiedler15example} a hierarchy of influence patterns was observed in several examples, on a purely phenomenological level and without deeper mathematical justification. Most notably, a simplified version of a model by Ishii and coworkers, \cite{ishii}, on the tricarboxylic citric acid cycle (TCAC) was studied. The paper \cite{fiedler15example} concluded:

\begin{quotation}
\textit{... our explanation of hierarchy patterns is still intuitive.
Only in the monomolecular case, so far, are we able to prove the
observed transitivity of the influence relation \emph{perturbation of
reaction $j^*$ implies flux change in reaction} $j'$, $j^* \leadsto j'$. 
This transitivity
underlies the hierarchy of flux response patterns and their relation
to directed cycles in the reaction network. See \cite{fiedler15mono} 
for the mathematical details which are beyond the scope of
our present exposition.}
\end{quotation}

The mathematical companion paper \cite{fiedler15mono}, however, did not yet meet our mathematical objectives. It served as just a first attempt towards a comprehensive mathematical understanding of the sparse anecdotal evidence compiled in \cite{fiedler15example}. Mathematical results were limited to the monomolecular case, where each reaction just converts a single metabolite. Bimolecular reactions were excluded, as were reactions with more than a single metabolite as a product. Although of substantial mathematical interest, the restriction to monomolecular reactions still excluded most biologically relevant examples. In particular, the original challenge by the TCA cycles of Ishii and Nakahigashi, \cite{ishii, ishii2}, had not yet been met.

Only in the present paper, for the first time, are we able to present a comprehensive and mathematically founded theory which explains, and proves, the appearance of sensitivity patterns and hierarchies of steady state responses in multimolecular metabolic networks.
The full mathematical background of a theoretical example from \cite{fiedler15example} is developed in section~\ref{sec3:theoex} below. Three variants of the Ishii TCAC metabolism \cite{ishii} will be compared in section~\ref{sec7:extca}, along with a further modification by Nakahigashi and coworkers, \cite {ishii2}, in the light of our present mathematical understanding.

In section~\ref{sec2:mainres} we present our main mathematical results. To prepare, we describe our general mathematical setting next.
We strongly recommend \cite{feinberg95, palsson05} for a general background.

Let the labels $j = 1, \ldots, E$ enumerate the total number $E$ of reactions in the network. Let $\mathcal{E} = \{1, \dots ,E\}$ denote the set of all reactions. Similarly, let the labels $m = 1,\ldots, M$ enumerate the total number $M$ of metabolites in the network. Let $\mathcal{M} = \{X_1, \dots , X_M\}$ denote the set of all metabolites. 
Mimicking standard chemical notation, a \emph{reaction network} is given by $E$ reactions
	\begin{equation}
	j: \quad y_1^j X_1 + \ldots + y_M^j X_M \ \longrightarrow \ 
	\bar{y}_1^j X_1 + \ldots + \bar{y}_M^j X_M\,.
	\label{eq:1.1}
	\end{equation}
The components of the \emph{stoichiometric coefficient} vectors $y^j, \bar{y}^j \in \mathbb{R}^M$ are nonnegative, and usually integer.
Reversible reactions are, not required but, admitted and will be listed as two separate reactions.
For notational convenience we will frequently identify $\mathcal{M}$ with its index set $\{1, \ldots, M\}$. We will even write $y^j, \bar{y}^j \in \mathcal{M}$, instead of $y^j, \bar{y}^j  \in\mathbb{R}^M$, to emphasize the supporting component set of vectors.
For subsets $\mathcal{M}_0 \subset \mathcal{M}$ we analogously abbreviate
	\begin{equation}
	\text{supp}\, y := \{ m\,|\,y_m \neq 0\} \subseteq \mathcal{M}_0\, \quad \text{by}\quad y \in \mathcal{M}_0\,, 
	\label{eq:1.2}
	\end{equation}
to denote the subspace of vectors $y \in \mathbb{R}^M$ which are supported on $\mathcal{M}_0$, only.
We call $m$ an \emph{input}, \emph{reactant}, or \emph{educt} of reaction $j$, if $y_m^j \neq 0$; in symbols:	\begin{equation}
	m \vdash j\,.
	\label{eq:1.3}
	\end{equation}
We use this notation even when $m$ \emph{catalyzes} $j$, i.e. when $\bar{y}_m^j = y_m^j$, and in presence of further inputs of reaction $j$.
\emph{Outputs}, or \emph{products} $m$ are given by $\bar{y}_m^j \neq 0$.
\emph{Feed reactions} only depend on external chemical input species which are provided at constant concentration levels, unaffected by the network. Thus feed reactions $j$ have $y^j =0$.
\emph{Exit reactions}, in contrast, only produce chemical output species which do not re-enter the network. Thus exit reactions $j$ have $\bar{y}^j =0$. See section \ref{sec3:theoex} for an easy example.

The \emph{stoichiometric matrix} $S$ is essential to derive the differential equation \eqref{eq:1.8}, below, for changes of metabolite concentrations in the network \eqref{eq:1.1}.
The $M \times E$ matrix $S$ is defined by
	\begin{equation}
	\begin{aligned}
	S: \quad \mathcal{E} &\longrightarrow \mathcal{M}\\
	e_j &\longmapsto \bar{y}^j -y^j
	\end{aligned}
	\label{eq:1.4}
	\end{equation}
where $e_j \in \mathcal{E}$, alias $\mathbb{R}^E$, denotes the $j$-th unit vector.
In other words, the columns $S^j$ of the stoichiometric matrix $S$ are simply the differences of the stoichiometric output and input vectors $\bar{y}^j$ and $y^j$.
The usually nonlinear \emph{reaction rate} $r_j$, at which reaction $j$ occurs per time unit, depends on the \emph{concentrations} $x_m$ of the metabolites $X_m$.
Evidently $r_j = r_j(x)$ only depends on the input metabolites of reaction $j$. In symbols, the partial derivatives $r_{jm}$ of the reaction rates $r_j$ satisfy
	\begin{equation}
	r_{jm} := \partial_{x_m} r_j(x) =0\,,\quad
	\text{unless} \quad m \vdash j\,.
	\label{eq:1.5}
	\end{equation}
In other words, the supports satisfy
	\begin{equation}
	(r_{jm})_{m\in\mathcal{M}} \in \text{Inputs}\, (j):=
	\{m \in\mathcal{M}\,|\, m\vdash j\}\,,
	\label{eq:1.6}
	\end{equation}
for each reaction $j\in \mathcal{E}$, by our convenient abuse (\ref{eq:1.2}) of notation.
The time evolution of the metabolite concentrations $x_m(t)$ is then  given by the coupled nonlinear ODE system
	\begin{equation}
	\dot{x}_m(t) = \sum_{j\in \mathcal{E}} (\bar{y}^j_m - y^j_m)\, r_j
	\big(x(t)\big)\,,
	\label{eq:1.7}
	\end{equation}
for $m=1,\ldots,M$, under isothermal conditions for the reaction rates. In vector notation with $x=(x_1, \dots , x_M)$ and $r= (r_1, \dots, r_E)$ this reads
	\begin{equation}
	\dot{x} = S r(x)\,.
	\label{eq:1.8}
	\end{equation}
For simplicity of presentation, we will assume
	\begin{equation}
	S \; \text{has full rank}\; M\, ,
	\label{eq:1.9}
	\end{equation}		
throughout our paper. This excludes cokernel of $S$, alias \emph{stoichiometric subspaces}. These are affine linear subspaces which are time invariant under the flow  of $x(t)$. In other words, we exclude trivially conserved linear combinations of the metabolite concentrations $x_m(t)$. For further comments on the case of stoichiometric subspaces we refer to the discussion in section \ref{sec:disc}.

\emph{Steady states} $x$ are time-independent solutions of \eqref{eq:1.8}, i.e.
	\begin{equation}
	0 = S r(x)\,.
	\label{eq:1.10}
	\end{equation}
Considerable effort has gone into the study of existence, uniqueness, and possible multiplicity of steady state solutions $x$ of \eqref{eq:1.10}.
See in particular Feinberg's pioneering work \cite{feinberg95}, his rather advanced recent results \cite{shinarfeinberg}, and the many references there.
In the present paper we do not address this question, at all.
Rather, \emph{we assume existence of a steady state} $x$ \emph{throughout this paper}.
We neither assume uniqueness, nor stability, of this steady state. We focus on the following \emph{central question}:
	\begin{equation}
	\begin{aligned}
	&\textit{How does a given steady state}\; x\;
	\textit{respond to perturbations}\\
	&\textit{of any 
	particular reaction rate}\; r_{j^*},\; 
	\textit{qualitatively}?
	\end{aligned}
	\label{eq:1.11}
	\end{equation}
To define and study this response sensitivity of steady states $x$ with respect to small external changes of any reaction rate $r_{j^*}$ we introduce formal reaction rate parameters $\varepsilon_j$ as
	\begin{equation}
	r_j = r_j (\varepsilon_j, x)\,.
	\label{eq:1.12}
	\end{equation}
We also write $r=r(\varepsilon, x)$ on the right hand side of \eqref{eq:1.10}.
For extreme generality, we could choose $\varepsilon_j$ to parametrize the functions $r_j$, themselves.
More modestly, we might think of just a rate coefficient like $r_j(\varepsilon_j,x)$:= $(1+\varepsilon_j)r_j(x)$.

The standard implicit function theorem in a $C^1$-setting immediately answers the response question -- albeit rather abstractly and under mild nondegeneracy assumptions.
Given a reference steady state solution $x(0)$, for $\varepsilon=0$, we obtain a unique differentiable family of solutions $x(\varepsilon)$, for sufficiently small $|\varepsilon|$, such that
	\begin{equation}
	0=S \,\partial_{\varepsilon_{j^*}} r 
	+S R \,\partial_{\varepsilon_{j^*}}
	x(\varepsilon)
	\label{eq:1.13}
	\end{equation}
holds for the partial derivatives.
Here the \emph{rate matrix} $R =(r_{jm})_{j\in \mathcal{E}, \, m \in \mathcal{M}}$ 
is the Jacobian matrix of the rate function vector $r(\varepsilon, x)$ with respect to $x$; see \eqref{eq:1.5} for the definition of the partial derivatives $r_{jm}$ of $r_j$.
For the partial derivatives of the rate function vector $r$ with respect to the parameters $\varepsilon$ it is sufficient to consider unit vectors
	\begin{equation}
	\partial_{\varepsilon_{j^*}} r
	= e_{j^*}\,;
	\label{eq:1.14}
	\end{equation}
all other cases can be derived from that.
We call
	\begin{equation}
	\delta x_m^{j^*}
	:= \partial_{\varepsilon_{j^*}} x_m
	\label{eq:1.15}
	\end{equation}
\emph{the sensitivity of metabolite} $m$ \emph{with respect to rate changes of} $j^*$.
Similarly, we call
	\begin{equation}
	\Phi_j^{j^*} := \delta_{j^*j}
	+ (R\delta x^{j^*})_j
	\label{eq:1.16}
	\end{equation}
the \emph{sensitivity of reaction flux} $j$ \emph{with respect to rate changes of} $j^*$.
The Kronecker symbol $\delta_{{j^*}j}$ accounts for the externally forced flux change by variation of the parameter $\varepsilon_{j^*}$ itself.
The reaction-induced term $R \delta x$ may, or may not, counteract or even annihilate this forcing, depending on the effected metabolite changes $\delta x$.

In vector notation, \eqref{eq:1.13} becomes the \emph{flux balance}
	\begin{equation}
	0 = S \Phi^{j^*}\,.
	\label{eq:1.17}
	\end{equation}

Our qualitative answer to the central sensitivity question \eqref{eq:1.11} will distinguish those fluxes and metabolites $\beta = j,\,m$ with nonzero response to the external rate change $j^*$, from those with zero response, i.e. without any response at all.
Let $\beta \in \mathcal{E} \cup \mathcal{M}$ denote any flux or metabolite.
We recall from \eqref{eq:1.18} that $j^*$ \emph{influences} $\beta$, in symbols $j^* \leadsto \beta$, if the component $\beta$ of the \emph{response vector} $(\Phi^{j^*}, \delta x^{j^*}) \in \mathcal{E} \times \mathcal{M}$ is nonzero. 

The five theoretical results of theorems~\ref{thm:2.1}--\ref{thm:2.7} below will investigate the influence relation \eqref{eq:1.18} in detail. Our first three results are reminiscent of \cite{fiedler15mono}, but now hold in the general multimolecular setting. Our results are based on the notion of a \emph{child selection}
	\begin{equation}
	J:\ \mathcal{M} \longrightarrow \mathcal{E}
	\label{eq:2.2}
	\end{equation}
which we define as follows.
We require injectivity of $J$, and we require each metabolite $m \in \mathcal{M}$ to be an input of the  selected reaction $J(m) \in \mathcal{E}$.
In symbols:
	\begin{equation}
	m \vdash J(m)\,,
	\label{eq:2.3}
	\end{equation}
for all $m \in \mathcal{M}$.
Occasionally we call $m$ a \emph{mother}, or \emph{parent}, of the \emph{child} $j$ if $m \vdash j$, i.e. $r_{jm} \not \equiv 0$.
Note that the same child $j$ may have two or more ``mothers'' $m$, i.e. input metabolites $m \vdash j$.
A \emph{single child} $j^*$, for example, possesses at least one mother $m^* \vdash j^*$ which has no other children, besides $j^*$ itself.
Then any child selection $J$ must select that single child $j^* = J(m^*)$ of the mother metabolite $m^*$.

Theorem~\ref{thm:2.1} addresses the standard requirement of a nonzero Jacobian
	\begin{equation}
	\det (SR) \neq 0\,,
	\label{eq:1.19}
	\end{equation}
in the standard implicit function theorem; see \eqref{eq:1.13}. 
We recall here that $S$ denotes the stoichiometric matrix and $R$ denotes the rate matrix.
We call the reaction network \eqref{eq:1.1} \emph{regular} at the steady state $x$ if \eqref{eq:1.19} holds there.
Theorem~\ref{thm:2.1} asserts that the nondegeneracy assumption (\ref{eq:1.19}) holds, algebraically, if and only if there exists a child selection $J:\mathcal{M} \rightarrow \mathcal{E}$ such that the columns $J(\mathcal{M})$ of the stoichiometric matrix $S$ form an invertible $M\times M$ block. This eliminates all consideration of the particular reaction rates $r(x)$. Indeed the condition on the child selection $J$ only involves conditions on the arrows $j=J(m)$ in the reaction network (\ref{eq:1.1}), as selected by $J$, and on the stoichiometric coefficients (\ref{eq:1.4}) associated to these arrows. 

Well, can that be true? Nowhere did we even bother to exclude the disastrous case $R=0$ of all constant reaction rates, which leaves $x$ undetermined. \emph{We consider} $\det SR$ \emph{as a polynomial expression in the nontrivial partial derivatives} $r_{jm}$ of the reaction rate functions $r_j$, evaluated at the steady state $x$.
We call an expression \emph{nonzero, algebraically,} if it is nonzero as a polynomial, or rational function, in  the real variables $r_{jm}$, taken over all metabolite inputs $m \vdash j$ of the reactions $j$.
It is in this precise sense, how influence \eqref{eq:1.18} and nondegeneracy \eqref{eq:1.19} are considered to hold.
In particular the set of $r_{jm}$ where our assertions may fail to hold is an algebraic variety of codimension at least one in the space of all real $r_{jm}$ with $m \vdash j$.
In this precise sense our results hold true for \emph{generic rate functions} $r_j$. Thus theorem~\ref{thm:2.1} clarifies the relation between the nondegeneracy assumption \eqref{eq:1.19} and child selections $J$ in the network. See section \ref{sec3:theoex} for a simple explicit example.

As a caveat we have to issue a warning concerning \emph{mass action kinetics}, defined by
	\begin{equation}
	r_j(x) = k_jx^{y^j} 
	:= k_j x_1^{y_1^j} \cdot \ldots \cdot x_M^{y_M^j}\,.
	\label{eq:1.20}
	\end{equation}
Here $y_m^j$ are nonnegative integers, and only a single parameter $k_j$ is available for each reaction.
The partial derivatives
	\begin{equation}
	r_{jm} = y_m^j r_j / x_m	
	\label{eq:1.21}
	\end{equation}
are closely related to the steady state rates $r_j$ themselves -- too closely, in fact, to be considered algebraically independent.
\emph{Therefore our theory does not apply to pure mass action kinetics}.
Slightly richer classes like Michaelis-Menten, \, Langmuir-Hinshelwood, or other kinetics where each participating species $x_m$, $m \vdash j$ enters with an individual kinetic coefficient, fall well within the scope of our algebraic results.

Theorem~\ref{thm:2.2} clarifies flux influence $j^* \leadsto j'$, in the same algebraic spirit. Algebraically nonzero self-influence $j^* \leadsto j^*$ occurs, if and only if there exists a child selection $J:\mathcal{M} \rightarrow \mathcal{E}$ such that $j^* \not\in J(\mathcal{M})$ and the columns $J(\mathcal{M})$ of the stoichiometric matrix $S$ form an invertible $M\times M$ block. Similarly, algebraically nonzero influence $j^* \leadsto j'\neq j^*$ occurs, if and only if there exists a child selection $J$ such that $j^* \not\in J(\mathcal{M}) \ni j'$ and the columns $\{j^*\}\cup J(\mathcal{M})\smallsetminus \{j'\}$ of the stoichiometric matrix $S$ form an invertible $M\times M$ block. Note how the influencing column $j^*$ has been swapped in, to replace the influenced column $j'$ of the stoichiometric matrix $S$.

Theorem~\ref{thm:2.3} clarifies metabolite influence $j^* \leadsto m'$, again in the same algebraic spirit. Algebraically nonzero metabolite influence $j^* \leadsto m'$ occurs, if and only if there exists a child selection $J:\mathcal{M}\smallsetminus \{m'\} \rightarrow \mathcal{E}$, on the remaining metabolites, such that $j^* \not\in J(\mathcal{M}\smallsetminus \{m'\})$ and the swapped columns $\{j^*\}\cup J(\mathcal{M}\smallsetminus \{m'\})$ of the stoichiometric matrix $S$ form an invertible $M\times M$ block. Here the influenced metabolite $m'$ has been swapped out of consideration, and the range of the remaining partial child selection $J$ has been augmented by the influencing column $j^*$ of the stoichiometric matrix $S$.

Transitivity theorem~\ref{thm:2.4}(ii) considers any 2-step chain of influences $j^* \leadsto j \leadsto \beta$, with either a reaction $\beta = j'$ or a metabolite $\beta = m'$ as the terminating target. As a consequence, direct influence $j^* \leadsto \beta$ is established. Compare \eqref{eq:1.18} and \eqref{eq:1.18trans}. Indeed, only this \emph{transitivity of influence} justifies the notion of a \emph{hierarchy of influence}. For example, only by transitivity of flux influences can we assert that any finite chain $j^* \leadsto j_1 \leadsto j_2 \leadsto \ldots \leadsto j_n$ of successive flux influences  amounts to a direct flux influence all along the chain. Indeed a rate perturbation of the first element $j^*$, in the influence chain, changes the reaction fluxes all the way, down to the very last terminating target $j_n$ of the whole chain.

In particular transitivity theorem~\ref{thm:2.4} allows us to summarize all algebraically nonzero flux responses $ j^* \leadsto j' \in \mathcal{E}$ as a directed acyclic \emph{influence graph} $\mathcal{F}$. The vertices of $\mathcal{F}$ are the lumped subsets of reactions $j \in \mathcal{E}$ which mutually influence each other; see theorem~\ref{thm:2.7}.
The responses of metabolites $\beta = m' \in \mathcal{M}$ appear as annotations of the reaction vertices in the influence graph.
See also figs.~\ref{fig:3.1} and \ref{fig:vanilla-tca} below for examples.

The remaining sections are organized as follows. Section \ref{sec3:theoex} illustrates the main results of section \ref{sec2:mainres} by two examples. First we treat the simplest case $E=M$ of a minimal number $E$ of reactions, given that the stoichiometric $M\times E$ matrix $S$ possesses full rank $M$. Any flux influences turn out to be absent in that case; see \eqref{eq:3.8}. Moreover we give a detailed mathematical account of the simple bimolecular theoretical example from \cite{fiedler15example} which we had not been able to treat in \cite{fiedler15mono}, due to the monomolecular restriction.

Algebraic nondegeneracy $\det\, SR \neq 0$ of networks, the flux influence relation $j^*\leadsto j'$, and the metabolite influence relation $j^* \leadsto m'$ as stated in theorems \ref{thm:2.1} -- \ref{thm:2.3} involve some analysis of child selections $J: \mathcal{E} \rightarrow \mathcal{M}$. The theorems are proved in section \ref{sec4:proofs}. 
The augmenticity section \ref{secB4:aug} addresses the question of the fate of influence relations, and influence regions, under modifications of the network. Specifically we augment an existing network by additional reactions among the existing metabolites and/or the addition of new metabolites. See in particular theorem \ref{thm:B4.1}. Similiarly to the first examples in section \ref{sec3:theoex}, this illustrates theorems \ref{thm:2.1} -- \ref{thm:2.3}  in a general context. 

In section \ref{sec5:trans} we prove transitivity theorem \ref{thm:2.4}, which is central to all claims on hierarchy of influences. Although transitivity involves nondegeneracy $\det\, SR \neq 0$, as a prerequisite, our proof is not based on child selections directly.

Symbolic packages would typically involve terms of a complexity which grows exponentially with network size. Our practical computational approach to influence graphs, as outlined in section \ref{sect:algo}, is based on integer arithmetic modulo large primes $p$, instead. The reaction coefficients $r_{jm}$ are chosen randomly $\mathrm{mod}\, p$. 

In section \ref{sec7:extca} our approach is brought to bear on the original Ishii TCAC example \cite{ishii} and the Nakahigashi augmentation \cite{ishii2}, in several variants. In the spirit of augmenticity section \ref{secB4:aug}, it turns out that the choice of exit reactions deserves particular attention.

We conclude with a discussion, in section \ref{sec:disc}. 
We briefly address the proper adaptation to large perturbations of reaction rates, as required by knockout experiments and in control settings.
Existence and multiplicity of steady states is a related issue. 
Although this case did not appear in our present examples, we also comment on the appearance of stoichiometric subspaces.
Revisiting augmenticity, as discussed in section \ref{sec:B4:aug}, reveals a cautioning menetekel which indicates how monomolecular exit reactions for all metabolites may destroy all hierarchic influence structure. We also point at beautiful recent progress concerning monomolecular networks.
We then turn to the elegant upper estimates \cite{OkadaFiedlerMochizuki15, OkadaMochizuki16} by Okada on the influence regions within certain subnetworks, in a multimolecular setting. 
A glimpse at the beautiful matroid results by Murota, as summarized in \cite{murota2009matrices}, concludes the paper: in pioneering work, Murota has established response patterns for layered matrices more than three decades ago.


\textbf{Acknowledgments.} 
We are greatly indebted to Hiroshi~Matano and Nicola~Vas-sena for their encouraging and helpful comments on mono- and multimolecular reactions.
To Hiroshi~Matano we owe the emphasis on the Cauchy-Binet formula in section~\ref{sec4:proofs}.
Marty~Feinberg has patiently explained his many groundbreaking results on existence, uniqueness, and multiplicity of steady states.
Takashi~Okada has generously explained and shared his elegant upper estimates on influence regions, as discussed in section~\ref{sec:disc}.
Anna~Karnauhova has drawn the TCA cycles of section \ref{sec7:extca} for us, with artistry and taste.
Ulrike~Geiger has tirelessly typeset and corrected several versions of the first manuscript.
This work has been supported by the Deutsche Forschungsgemeinschaft, SFB 910 ``Control of Self-Organizing Nonlinear Systems''.


\section{Main results}
\label{sec2:mainres}

In this section we present our main results, theorems~\ref{thm:2.1}--\ref{thm:2.7}.
For background notation and an outline see section~\ref{sec1:intro}.
Throughout this paper, and in all theorems, we assume surjectivity \eqref{eq:1.9} of the stoichiometric matrix $S$ defined in \eqref{eq:1.4}.
After our analysis of this condition in theorem~\ref{thm:2.1} we also assume  the reaction network \eqref{eq:1.1} to be regular, i.e. the nondegeneracy condition
	\begin{equation}
	\det (SR)\neq 0
	\label{eq:2.1}
	\end{equation}
holds, algebraically, for the rate Jacobian $SR$ which is the product of the stoichiometric matix $S$ with the rate matrix $R$; see also \eqref{eq:1.3} and  \eqref{eq:1.19}.
We repeat and emphasize that, here and below, any nonzero quantities in assumptions or conclusions are understood in the algebraic, polynomial sense as explained in the introduction, section \ref{sec1:intro}.

Recall that we have required full rank of $S$ in \eqref{eq:1.9}, i.e. surjectivity $\mathcal{M} = \text{range}\; S$.
Consider any subset $\mathcal{E}'$ of $\mathcal{E}$.
We say that $\mathcal{E}'$ \emph{selects an} $S$-\emph{basis} of the stoichiometric matrix $S$: $\mathcal{E} \rightarrow \mathcal{M}$, if the columns $\mathcal{E}'$ of $S$ are a basis of range $S$.
This means $|\mathcal{E}'| = |\mathcal{M}| = M$ and $\det S^{\mathcal{E}'} \neq 0$ for the square minor $S^{\mathcal{E}'}$ of $S$ defined by the columns $\mathcal{E}'$ of $S$.
In other words,
	\begin{equation}
	\ker S^{\mathcal{E}'} = \ker S \cap \mathcal{E}' = \{ 0\}\,.
	\label{eq:2.4a}
	\end{equation}
According to our notation convention \eqref{eq:1.2} this means that $S$: $\mathbb{R}^E \rightarrow \mathbb{R}^M$ does not possess any nontrivial kernel vector supported only in $\mathcal{E}' \subseteq \mathcal{E}$.

\begin{thm}
Let $S$ be surjective, as required in \eqref{eq:1.9}.
Then the reaction network \eqref{eq:1.1} is regular, algebraically, i.e. the nondegeneracy condition $\det(SR) \neq 0$ of \eqref{eq:2.1} holds, algebraically, if and only if there exists a child selection $J$: $\mathcal{M} \rightarrow\mathcal{E}$ such that $J(\mathcal{M})$ selects an $S$-basis, i.e.
	\begin{equation}
	\ker S \cap J(\mathcal{M}) = \{ 0\}\,.
	\label{eq:2.4b}
	\end{equation}
\label{thm:2.1}
\end{thm}

See \eqref{eq:2.2}, \eqref{eq:2.3} for the definition of child selections $J$. Henceforth, and throughout the rest of the paper, we will assume $\det (SR) \neq 0$, i.e. the existence of such a kernel-free child selection $J$.

The meaning of the child selection $J$ will become more evident in the proof; see section~\ref{sec4:proofs}.
We will in fact show that $\det(SR)$ can be written as a polynomial
	\begin{equation}
	\det (SR) = \sum_J a_J \mathbf{r}^J\,.
	\label{eq:2.4c}
	\end{equation}
Here the sum runs over all child selections $J$, and $\mathbf{r}^J$ abbreviates the monomial
	\begin{equation}
	\mathbf{r}^J := \prod_{m \in \mathcal{M}} r_{J(m),\,m}\ .
	\label{eq:2.4d}
	\end{equation}
Up to sign, the coefficient $a_J$ abbreviates the determinant
	\begin{equation}
	a_J = \pm \det S^{J(\mathcal{M})}\,,
	\label{eq:2.4e}
	\end{equation}
where $S^{J(\mathcal{M})}$ is the $ M \times M$ minor of the stoichiometric $S$-columns $J(\mathcal{M})$.
Of course condition \eqref{eq:2.4b} is equivalent to a (truly) nonzero 	coefficient $a_J$ of the algebraic monomial $\mathbf{r}^J$.

Let us briefly illustrate the role of single children $j^*$ in our nondegeneracy setting $\det (SR) \neq 0$.
We call a reaction $j^* \in \mathcal{E}$ a \emph{single child}, if there exists an input metabolite $m^* \vdash j^*$ such that $m^* \vdash j$ implies $j=j^*$. 
In other words, the single child $j^*$ is the only child reaction of some mother metabolite $m^*$ of $j^*$. 
In such a case
	\begin{equation}
	J(m^*) = j^*
	\label{eq:2.4f}
	\end{equation}		
holds for any child selection $J$.

Let the reaction $j^*$ be a single child.
Our standing assumption $\det (SR) \neq 0$ and theorem~\ref{thm:2.1} then imply the existence of a child selection $J$.
By definition, the child selection $J$ must be injective. 
Hence, by \eqref{eq:2.4f}, there cannot exist any other single-child parent $m \neq m^*$ of $j^*$, i.e. $m \vdash j^*$, with the same reaction $j^*$ as its own single child.
Therefore the mother $m^*$, of which $j^*$ is a single child, is determined uniquely by the single child $j^*$ in our setting of $\det (SR) \neq 0$.

The next theorem characterizes flux influences $j^* \leadsto j'$ in terms of swapped child selections.
We recall definition \eqref{eq:1.16} of the flux response $\Phi_{j'}^{j^*}$, and how $j^* \leadsto j'$ means $\Phi_{j'}^{j^*} \neq 0$, algebraically.	
	
\begin{thm}
Assume $\det (SR) \neq 0$ holds, algebraically, as asserted by theorem~\ref{thm:2.1}.
Then flux influence is characterized as follows.
\begin{itemize}
	\item[(i)] Self-influence $j^* \leadsto j^*$ occurs, if and only if there exists a child selection \mbox{$J$: $\mathcal{M} \rightarrow \mathcal{E} \smallsetminus \{ j^*\}$} such that $J(\mathcal{M})$ selects an $S$-basis, i.e.
	\begin{equation}
	j^* \not\in J(\mathcal{M}) \quad \mathrm{and}
	\quad \ker S \cap J(\mathcal{M}) = \{ 0\}\,.
	\label{eq:2.5}
	\end{equation}
	\item[(ii)] Influence $j^* \leadsto j' \neq j^*$ occurs, if and only if there exists a child selection \mbox{$J$: $ \mathcal{M} \rightarrow \mathcal{E}$} such that $j^* \not\in J(\mathcal{M}) \ni j'$ and the \emph{swapped set} $\{ j^*\} \cup J(\mathcal{M}) \smallsetminus \{ j'\}$ selects an $S$-basis, i.e.
	\begin{equation}
	j^* \not\in J(\mathcal{M}) \ni j' \quad \mathrm{and}
	\quad \ker S \cap \big( \{ j^*\} \cup J(\mathcal{M})
	\smallsetminus \{j'\}\big) = \{0\}\,.
	\label{eq:2.6}
	\end{equation}
\end{itemize}
Both cases may occur for the same perturbation $j^*$.
\label{thm:2.2}
\end{thm}

To illustrate theorem \ref{thm:2.2} we observe that \emph{single children have no flux influence}.
Indeed let $m^* \vdash j^*$ be the unique single-child mother of the single child $j^*$.
Then \eqref{eq:2.4f} implies $J(m^*) = j^*$.
This contradicts both \eqref{eq:2.5} and \eqref{eq:2.6}, and hence prohibits flux influence $j^* \leadsto j$ of the single child $ j^*$ on any $j \in \mathcal{E}$, including itself.
We will see many examples of this simple principle later.

To characterize metabolite influences $j^* \leadsto m' \in \mathcal{M}$ we have to define \emph{partial child selections} $J^\vee$: $ \mathcal{M} \smallsetminus \{ m'\} \rightarrow \mathcal{E}$.
We require injectivity and the input property $m \vdash J(m)$ as in \eqref{eq:2.3}, verbatim but just for all metabolites $m \neq m'$, of course.

\begin{thm}
Assume $\det (SR)\neq 0$ holds, algebraically, as asserted by theorem~\ref{thm:2.1}.
Then metabolite influence $j^* \leadsto m'$ occurs, if and only if there exists a partial child selection $ J^\vee$: $ \mathcal{M} \smallsetminus \{ m'\} \rightarrow \mathcal{E} \smallsetminus \{j^*\}$, such that the \emph{augmented reaction set} $\{j^*\} \cup J^\vee (\mathcal{M} \smallsetminus \{m'\})$ selects an $S$-basis, i.e.
	\begin{equation}
	j^* \not\in J^\vee( \mathcal{M} \smallsetminus \{m'\}) \quad
	\mathrm{and} \quad \ker S \cap \big(\{ j^*\} \cup 
	J^\vee (\mathcal{M} \smallsetminus \{ m'\})\big) = \{0\}\,.
	\label{eq:2.7}
	\end{equation}
\label{thm:2.3}
\end{thm}

Again the above \emph{single child case} $m^* \vdash j^*$ with unique single-child mother $m^*$ of the single child $j^*$ is instructive.
We claim that \emph{single children only influence their own mother}.
Indeed consider $m' = m^*$, first.
We have assumed $\det (SR) \neq 0$.
By theorem~\ref{thm:2.1} this provides a kernel-free child selection $J$: $\mathcal{M} \rightarrow\mathcal{E}$, see \eqref{eq:2.3}.
Moreover $J(m^*) = j^*$, by \eqref{eq:2.4f}.
Define the partial child selection $J^\vee$ as the restriction of $J$ to $\mathcal{M} \smallsetminus \{m^*\}$.
Then \eqref{eq:2.7} follows from \eqref{eq:2.4b}.
Hence $j^* \leadsto m^*$ influences its unique single-child mother $m^*$, by theorem \ref{thm:2.3}.
Next consider $m' \neq m^*$.
Then $m^* \in \mathcal{M} \smallsetminus \{m'\}$ implies $j^* = J^\vee(m^*) \in J^\vee (\mathcal{M} \smallsetminus \{ m'\})$,  for any partial child selection $J^\vee$ on $\mathcal{M} \smallsetminus \{m'\}$, again by \eqref{eq:2.4f}.
This prevents any metabolite influence of the single child $j^*$ of $m^*$, other than $ j^* \leadsto m^*$.

In the introduction we have emphasized how transitivity \eqref{eq:1.18trans} of flux influence \eqref{eq:1.18} is at the heart of any notion like ``hierarchy of influence'' in reaction networks. In the monomolecular case of \cite{fiedler15mono} we were able to show how transitivity of flux influence followed from a study of paths in the directed reaction graph with vertex set $\mathcal{M} \cup \{0\}$. See section \ref{sec:disc} for further discussion. We now formulate our
\emph{transitivity theorem} in a multimolecular setting. Instead of any reaction di-graph, our proof in section~\ref{sec5:trans} will involve direct differentiation with respect to an intermediate variable $r_{jm}$.

\begin{thm}
Assume $\det (SR) \neq 0$ holds, algebraically, as asserted in theorem~\ref{thm:2.1}.
Then the following two transitivity properties hold.
\begin{itemize}
	\item[(i)] 	Assume that $\alpha \in \mathcal{E} \cup \mathcal{M}$ influences an input metabolite $m$ of reaction $j$, and $j$ influences $	\beta \in \mathcal{E} \cup \mathcal{M}$. In symbols,
	\begin{equation}
	\alpha \leadsto m \vdash j \leadsto \beta\,.
	\label{eq:2.8}
	\end{equation}
	Then $\alpha \leadsto \beta$, that is, $\alpha$ influences $\beta$.
	\item[(ii)] Assume that $\alpha \in \mathcal{E} \cup \mathcal{M}$ influences reaction $j$, and $j$ influences $\beta \in \mathcal{E} \cup \mathcal{M}$. In symbols,
	\begin{equation}
	\alpha \leadsto j \leadsto \beta\,.
	\label{eq:2.9}
	\end{equation}
	Then $\alpha \leadsto \beta$, that is, $\alpha$ influences $\beta$.
\end{itemize}	
\label{thm:2.4}
\end{thm}

Consider transitivity theorem \ref{thm:2.4}(ii) for the special case of just reactions $\alpha = j^*$ and $\beta = j'$. Then the theorem establishes transitivity of flux influence, as claimed in \eqref{eq:1.18trans} above, for the general multimolecular case.

Although we have admitted metabolites $\alpha \in \mathcal{M}$ in our transitivity theorem, we have not even defined yet what an influence $m^* \leadsto m$ or $m^*  \leadsto j$ of a metabolite $m^*$  on a metabolite $m$ or a reaction $j$ is supposed to mean.
Loosely speaking, such influences describe an algebraically nonzero sensitivity response of the steady state $x$ under the addition of an artificial constant external feed of metabolite $m^*$.
For a precise definition see \eqref{eq:3.4},~\eqref{eq:3.5} below, and the discussion there.

Our next goal is a description of all \emph{flux influence sets}
	\begin{equation}
	I_{\mathcal{E}}(j^*)
	:= \{ j' \in \mathcal{E}\,|\, j^* \leadsto j'\}\,,
	\label{eq:2.10}
	\end{equation}
and of all \emph{metabolite influence sets}
	\begin{equation}
	I_{\mathcal{M}}(j^*) = \{ m' \in \mathcal{M}\,|\,j^* 
	\leadsto m'\}\,,
	\label{eq:2.11}
	\end{equation}
for rate perturbations of any single reaction $j^*$.
Of course theorems~\ref{thm:2.2} and~\ref{thm:2.3} characterize these sets.
For single children $m^* \vdash j^*$ we have seen $I_{\mathcal{E}} (j^*) =\emptyset$ and $I_{\mathcal{M}}(j^*)= \{m^*\}$.

Transitivity theorem~\ref{thm:2.4} suggests a \emph{hierarchy of influence} for the influence sets. To describe this hierarchy, we
first construct the directed \emph{pure flux influence graph} $F$ as follows.
We start from an \emph{equivalence relation} $\approx
$ on $j \in \mathcal{E}$, defined by artificial reflexivity
	\begin{equation}
	j \approx j\,, \quad \text{for all} \quad j\,,
	\label{eq:2.20}
	\end{equation}
and by mutual influence between different reactions $j_1 \neq j_2$:
	\begin{equation}
	j_1 \approx j_2 , \quad \text{if}
	\quad j_1 \leadsto j_2 \quad \text{and}
	\quad j_2 \leadsto j_1\,.
	\label{eq:2.21}
	\end{equation}
Note that reflexivity \eqref{eq:2.20} may, or may not, be founded on actual self-influence $j \leadsto j$.
Our definition generously glosses over this delicate point.

The equivalence classes of $\approx$ are the vertices of the pure flux influence graph $F$.
In general, we denote an equivalence class $\varphi$ by brackets
	\begin{equation}
	\varphi =\langle j_1, j_2, \ldots \rangle
	\label{eq:2.22}
\end{equation}		
and call it a \emph{(mutual) flux influence class}.
Note how any reaction $j \in \varphi$ actually influences itself, by transitivity, whenever the class $\varphi$ contains at least two elements.
The subtle case of single element equivalence classes, however, comes in two flavors.
We write such single element classes $\varphi$ of $j$ as
	\begin{equation}
	\varphi=
	\left\lbrace 
	\begin{aligned}
	\langle &j\rangle\,, 
	\qquad &\text{if}\quad j \leadsto j\,,\\
	&j\,,
	\qquad &\text{otherwise}\,.	
	\end{aligned}
	\right.
	\label{eq:2.23}
	\end{equation}
We simply omit the brackets $\langle \,\rangle$ in case $j$ does not influence itself.
This is the case where reflexivity $j \approx j$ was generously decreed by mathematical \emph{fiat}, in \eqref{eq:2.20}, in spite of lacking actual self-influence.

Transitivity defines a partial order on the equivalence classes, by actual directed influence.
We emphasize this important point. Abstractly, any block triangularization of a matrix can be
stylized into a purely formal ``partial order'' of the blocks. Only in the presence of a transitivity
result, however, the formal partial order becomes one of actual influence. Mathematically,
this would correspond to the additional assertion that the blocks in the triangularization
are fully occupied by nonzero elements.

The above partial order of actual flux influence can be expressed, equivalently, by a minimal finite directed graph $F$.
Directed paths run in the direction of, and imply, flux influence.
Since the vertices $\varphi$ are equivalence classes, $F$ does not possess any directed cycle.
More precisely, a directed path from one class vertex $\varphi^*$ to another class $\varphi'$ implies that each single $j^*$ in $\varphi^*$ influences each $j'$ in $\varphi'$ -- but there does not exist any single influence in the opposite direction.
Also, $j^*$ influences all $j'$ in its own class $\varphi^*$ -- except possibly itself; see \eqref{eq:2.23}.
Therefore the pure flux influence graph $F$, in the above notation, describes the flux influence set $I_{\mathcal{E}}(j^*)$ of a rate perturbation at any given reaction $j^*$.

To include the metabolite responses to $j^*$, i.e. the metabolite influence sets $I_{\mathcal{M}}(j^*)$, we define the \emph{(full) flux influence graph} $\mathcal{F}$, by a simple annotation at the vertices of $F$.
We keep all directed edges, as defined between the equivalence classes of the pure flux influence di-graph $F$.

To determine the metabolite influence sets $I_{\mathcal{M}}(j^*)$, we apply transitivity theorem~\ref{thm:2.4}(ii) with reactions $\alpha = j_1,\  j = j_2 \in \mathcal{E}, \ \beta = m' \in \mathcal{M}$, and observe
	\begin{equation}
	j_1 \leadsto j_2 \leadsto m' \quad \Longrightarrow \quad 
	j_1 \leadsto m'\,.
	\label{eq:2.24}
	\end{equation}
Together with flux transitivity \eqref{eq:1.18trans} this implies
	\begin{equation}
	\begin{aligned}
	j_1 \leadsto j_2  \quad &\Longrightarrow \quad 
	I_{\mathcal{M}}(j_1) 
	\supseteq I_{\mathcal{M}}(j_2)\,,\quad \text{and}\\
	j_1 \approx j_2 \quad &\Longrightarrow \quad
	I_{\mathcal{M}}(j_1) =
	I_{\mathcal{M}}(j_2)\,.
	\end{aligned}
	\label{eq:2.25}
	\end{equation}
Now fix any mutual flux influence class $\varphi^*$:= $\langle j^*, \ldots \rangle$ or $\varphi^*$:= $j^*$, i.e. any vertex of the pure flux influence graph $F$.
By \eqref{eq:2.25}, all $j^* \in \varphi^*$ share the same metabolic influence set
	\begin{equation}
	I_{\mathcal{M}}(j^*) = I _{\mathcal{M}}(\varphi^*)\,.
	\label{eq:2.26}
	\end{equation}
We define the, possibly empty, \emph{indirect metabolite influence set} $\mathcal{M}^{\neg d}(\varphi^*)$ as
	\begin{equation}
	\begin{aligned}
	\mathcal{M}^{\neg d}(\varphi^*)
	&=
	\mathcal{M}^{\neg d}(j^*) = \\
	&=\{m' \in \mathcal{M} \,|\,  j^* \leadsto  j'\leadsto m' \; \mathrm{for\ some} \; j' \not\in \varphi^*\}\,.
	\end{aligned}
	\label{eq:2.27a}
	\end{equation}
In other words, indirect influence of $j^* \in \varphi^*$ on $m'$ requires flux influence on an \emph{intermediary agent} $j'$ in another flux influence class $\varphi' \neq \varphi^*$ to exert its influence on $m$ via $j'$.
By transitivity \eqref{eq:2.24} that intermediary $j'$ mediates the (indirect) influence $j^* \leadsto m'$ This influence does not depend on the choice of the representatives $j^* \in \varphi^*$, $j' \in \varphi'$.
Conversely, the intermediary agent $j' \not\in \varphi^*$ of $j^*$ does not influence $j^* \in \varphi^*$.

Analogously we define the, possibly empty, complementary set $\mathcal{M}^d(\varphi^*)$ of \emph{direct metabolite influence}
	\begin{equation}
	\mathcal{M}^d(\varphi^*) =
	\mathcal{M}^d (j^*) :=
	I_{\mathcal{M}}(\varphi^*) \smallsetminus
	\mathcal{M}^{\neg d} (\varphi^*)	
	\label{eq:2.27b}
	\end{equation}		
to consist of all those metabolites $j^* \leadsto m'$ which are influenced by $j^* \in \varphi^*$, but cannot ever be influenced by any intermediary agent $j' \not\in \varphi^*$.
In particular we obtain the decompositions
	\begin{alignat}{2}
	I_{\mathcal{M}}(j^*) \,=\, I_{\mathcal{M}} &(\varphi^*)\,&=&\,
	\mathcal{M}^d(\varphi^*)\, \dot{\cup}\,
	\mathcal{M}^{\neg d}(\varphi^*)\,, \quad \text{and}\label{eq:2.28a}\\
	\mathcal{M}^{\neg d} &(\varphi^*) \,&=&\,
	\underset{j^* \leadsto j' \not\in \varphi^*}{\bigcup}
	\mathcal{M}^d(j')\,,
	\label{eq:2.28b}
	\end{alignat}
for $j^* \in \varphi^*$. By definition, the union in \eqref{eq:2.28a} is disjoint.
We may replace the intermediary agents $j'$ in \eqref{eq:2.28b} by a single representative in each of their mutual flux influence classes.

Note that the classes $\varphi$ of mutual flux influence form a partition of $\mathcal{E}$.
The same metabolite $m$, in contrast, may appear in several sets $\mathcal{M}^d(\varphi^*)$ of direct metabolite influence.
This is the case, if and only if there exist two distinct reactions $j_1, j_2$ such that neither influences the other, but each influences $m$.
In particular, the union in \eqref{eq:2.28b} need not be disjoint.

We define the \emph{(full) flux influence graph} $\mathcal{F}$ as follows.
The annotated vertices of $\mathcal{F}$ are given by the pairs
	\begin{equation}
	\varphi \mathcal{M}^d(\varphi)
	\label{eq:2.29}
	\end{equation}
of mutual flux influence classes $\varphi$, annotated by their direct metabolite influence sets $\mathcal{M}^d(\varphi)$; see \eqref{eq:2.22}, \eqref{eq:2.23} and \eqref{eq:2.27b}.
Edges are directed and coincide with the directed edges of the pure flux influence graph $F$ on the vertices $\varphi$.
Summarizing our above construction and discussion, we have proved the following theorem on the transitive hierarchic structure of flux influence.

\begin{thm}
Assume $\det (SR)\neq 0$ holds, algebraically, as asserted by theorem~\ref{thm:2.1}.
Define the full flux influence graph $\mathcal{F}$ as above, and let $\varphi^*$ be the flux influence class of $j^*$; see \eqref{eq:2.22}, \eqref{eq:2.23}.

Then a rate perturbation of reaction $j^*$ influences the flux of $j' \neq j^*$, i.e. $j^* \leadsto j' \neq j^*$, if and only if, either $j' \in \varphi^*$, or else there exists a directed path in $\mathcal{F}$ from the vertex $\varphi^* \mathcal{M}^d(\varphi^*)$ to the vertex $\varphi' \mathcal{M}^d(\varphi')$ of the flux influence class $j' \in \varphi' \neq \varphi^*$ of $j'$.
Equivalently, the same directed path runs from $\varphi^*$ to $\varphi'$ in the pure flux influence graph $F$.
Self-influence $j^* \leadsto j^*$ holds unless $\varphi^* =j^*$ with removed brackets; see \eqref{eq:2.23}.

A perturbation $j^*$ influences the metabolite $m' \in \mathcal{M}$; i.e. $j^* \leadsto m'$, if and only if,
\begin{itemize}
	\item[(i)] either, $m'\in \mathcal{M}^d(\varphi^*)$ is under the direct influence of the class $\varphi^*$ of $j^*$; see \eqref{eq:2.27b};
	\item[(ii)] or else, $m'\in \mathcal{M}^d(\varphi')$ is under the direct influence of some intermediary agent $j'$ in another class $\varphi' \neq \varphi^*$ such that $j^* \leadsto j'$; see \eqref{eq:2.27a}--\eqref{eq:2.28b}.
\end{itemize}

In the flux influence graph $\mathcal{F}$ this means that we pass from $j^*\in \varphi^*$, either, to $m' \in \mathcal{M}^d(\varphi^*)$ directly in the same vertex $\varphi^* \mathcal{M}^d(\varphi^*)$, or else, to $m'$ in the annotation $\mathcal{M}^d(\varphi')$ of another vertex $\varphi' \mathcal{M}^d(\varphi')$, via some directed path in $\mathcal{F}$.
\label{thm:2.7}
\end{thm}

We recall and emphasize that these nonzero influences are all generated, in unison and simultaneously, by flux perturbations at any single one reaction $j^*$ of the same class $\varphi^*$.


\section{Two theoretical examples}
\label{sec3:theoex}

By our standing assumption \eqref{eq:1.9}, the stoichiometric matrix $S$: $\mathcal{E} \rightarrow \mathcal{M}$ has full rank $M=|\mathcal{M}|$.
Therefore $E = |\mathcal{E}| \geq |\mathcal{M}|$:
the number $M$ of metabolites does not strictly exceed the number $E$ of reactions.
Our first example addresses the simplest case where $E=M$.
As announced in the introduction, our second example studies the (full) flux influence graph $\mathcal{F}$ for a simple hypothetical reaction network which is monomolecular, except for a single bimolecular reaction; here $E = 15 > 10 = M$.
That second example was first presented in \cite{fiedler15example} and had to be treated in an adhoc and case-by-case fashion, because the present multimolecular mathematical framework was not available at the time.
Both examples illustrate the workings of theorem~\ref{thm:2.7} about direct and indirect flux influence.

Here, and for all our proofs in sections~\ref{sec4:proofs}--\ref{sec5:trans} below, it will be convenient to rewrite the system \eqref{eq:1.16}, \eqref{eq:1.17} for the response vector $z^{j^*}$:= $(\Phi^{j^*}, \delta x^{j^*}) \in \mathcal{E} \cup \mathcal{M}$ in block matrix form as
	\begin{equation}
	B z^\alpha = -e_\alpha.
	\label{eq:3.1}
	\end{equation}
As always, we have padded the unit vector $e_\alpha = e_{j^*} \in \mathcal{E} \subseteq \mathcal{E} \cup \mathcal{M}$, for $\alpha = j^*$, with zeros in the $\mathcal{M}$-components. 
We have used the $E \times M$ block matrix
	\begin{equation}
	B:=
	\begin{pmatrix}
	-\mathrm{id_{\mathcal{E}}} & R\\S & 0
	\end{pmatrix}
	: \quad \mathcal{E} \cup 
	\mathcal{M} \longrightarrow \mathcal{E}\cup
	\mathcal{M}\,
	\label{eq:3.2}
	\end{equation}
in \eqref{eq:3.1}. Again $S$ denotes the stoichiometric matrix, and $R$ is the rate matrix. We recall that, by definition, we have
	\begin{equation}
	\alpha \leadsto \beta \quad
	 \Longleftrightarrow \quad
	 z_\beta^\alpha \neq 0\,,\ 
	 \text{algebraically}\,,
	\label{eq:3.3}
	\end{equation}
for any $\alpha \in \mathcal{E},\, \beta \in \mathcal{E} \cup \mathcal{M}$.

We digress briefly to consider metabolite ``perturbations'' $\alpha = m^* \in \mathcal{M}$, as well. 
\emph{Define} $z^{m^*}$ \emph{as the} \emph{(flux, metabolite)-response to an external perturbation of the metabolite} $m^*$.
In other words, we define
\begin{equation}
	m^* \leadsto \beta \quad
	 \Longleftrightarrow \quad
	 z_\beta^{m^*} \neq 0\,,\quad
	 \text{algebraically}\,.
	\label{eq:3.4}
	\end{equation}
From an applied point of view this means that we study the (infinitesimal) steady state response $\delta x^{m^*}$ of $x$, and the associated flux changes $\Phi^{m^*}$, to external feeds of $x_{m^*}$,
	\begin{equation}
	\dot{x} = S \mathbf{r} (x) +
	\varepsilon \,  e_{m^*}\,,
	\label{eq:3.5}
	\end{equation}
as the (infinitesimal) rate $\varepsilon$ of that feed varies.

After this metabolic digression we now present the case $E=|\mathcal{E}|=|\mathcal{M}|=M$ as our first example.
The matrices $S$ and $R$ are square, and
	\begin{equation}
	\det B=(-1)^E \det S \cdot
	\det R =(-1)^E \det (SR)\,.
	\label{eq:3.6}
	\end{equation}
Note $\det S \neq 0$, by our full rank assumption \eqref{eq:1.9}.

The steady state equation \eqref{eq:1.10} becomes $r(x) =0$, for $\det S \neq 0$.
If $r(x)>0$ for $ x>0$, componentwise, this precludes the existence of positive steady states $x>0$.
However, we did not impose any such positivity restrictions on $r(x)$.
Even in the monomolecular case, for example, a feed reaction $0 \rightarrow X_1$ in \eqref{eq:1.1} can be lumped with a subsequent forward reaction $X_1 \rightarrow X_2$ into a single reaction term like $r_1=r_1(x) = k_0 -k_1x_1$, at the expense of violating positivity $r_1 >0$.

The child selections $J$: $ \mathcal{M} \rightarrow\mathcal{E}$ of theorem~\ref{thm:2.1} readily appear in the evaluation of $\det R$. 
In fact,  \eqref{eq:2.4c}, \eqref{eq:2.4e} hold with
	\begin{equation}
	a_J = \text{sgn}\,J \cdot \det S\,,
	\label{eq:3.7}
	\end{equation}
if we view $J$ as a permutation of $|\mathcal{M}| = |\mathcal{E}|$ elements with signature $\text{sgn}\,J$.
Our nondegeneracy assumption $\det (SR) \neq 0$ of theorem \ref{thm:2.2} amounts to (algebraic) invertibility of $R$.

Theorem~\ref{thm:2.2} informs us that there are no influences $j^* \leadsto j' \in \mathcal{E}$ at all, since $j^* \not\in J(\mathcal{M}) = \mathcal{E}$ is impossible.
Of course this also follows directly, because $S \Phi^{j^*} =0$ in \eqref{eq:1.17} and $\det S \neq 0$ imply $\Phi^{j^*}=0$, for any $j^* \in \mathcal{E}$.
Thus the flux influence sets $I_\mathcal{E}(j^*)$  of \eqref{eq:2.10} are all empty.
The pure flux influence graph $F$ consists of the $E$ isolated vertices $j^* \in \mathcal{E}$, sadly without any edges.
The full flux influence graph $\mathcal{F}$, then, does not possess any edges either.
Therefore all metabolite influences $I_\mathcal{M}(j^*)$ are direct, i.e.
	\begin{equation}
	I_\mathcal{E} (j^*) = \{ {~} \}\,; \qquad
	I_\mathcal{M}(j^*) = \mathcal{M}^d(j^*)\,.
	\label{eq:3.8}
	\end{equation}
Theorem~\ref{thm:2.3} implies that $j^* \leadsto m'$, if and only if $J(m') = j^*$, for some child selection permutation $J$: $\mathcal{M} \rightarrow\mathcal{E}$.
Indeed $\ker S= \{0\} $.
Hence
	\begin{equation}
	\mathcal{M}^d(j^*) = 
	\{ J^{-1}(j^*)\, |
	\, J\; \text{is a child selection}\}\,.
	\label{eq:3.9}
	\end{equation}
Clearly the metabolite influence sets $\mathcal{M}^d(j^*)$ will not be disjoint, here, if the reaction network admits more than one child selection permutation $J$.

We now revisit the adhoc bimolecular example of \cite{fiedler15example} with 15~reactions $\mathcal{E} = \{ 1,\dots ,15\}$, 10~metabolites $\mathcal{M} = \{ \texttt{A}, \dots , \texttt{J}\}$, and a single bimolecular reaction
	\begin{equation}
	j = 12: \quad \texttt{G}+\texttt{H} \longrightarrow I\,,
	\label{eq:3.10}
	\end{equation}
see fig.~\ref{fig:3.1}(a).
In spite of this notation inherited from \cite{fiedler15example}, we do not plan to confuse metabolites \texttt{B,F,J} with the matrix $B$, the pure flux influence graph $F$, or child selections $J$.

\begin{figure}[]
\centering \includegraphics[width=\textwidth]{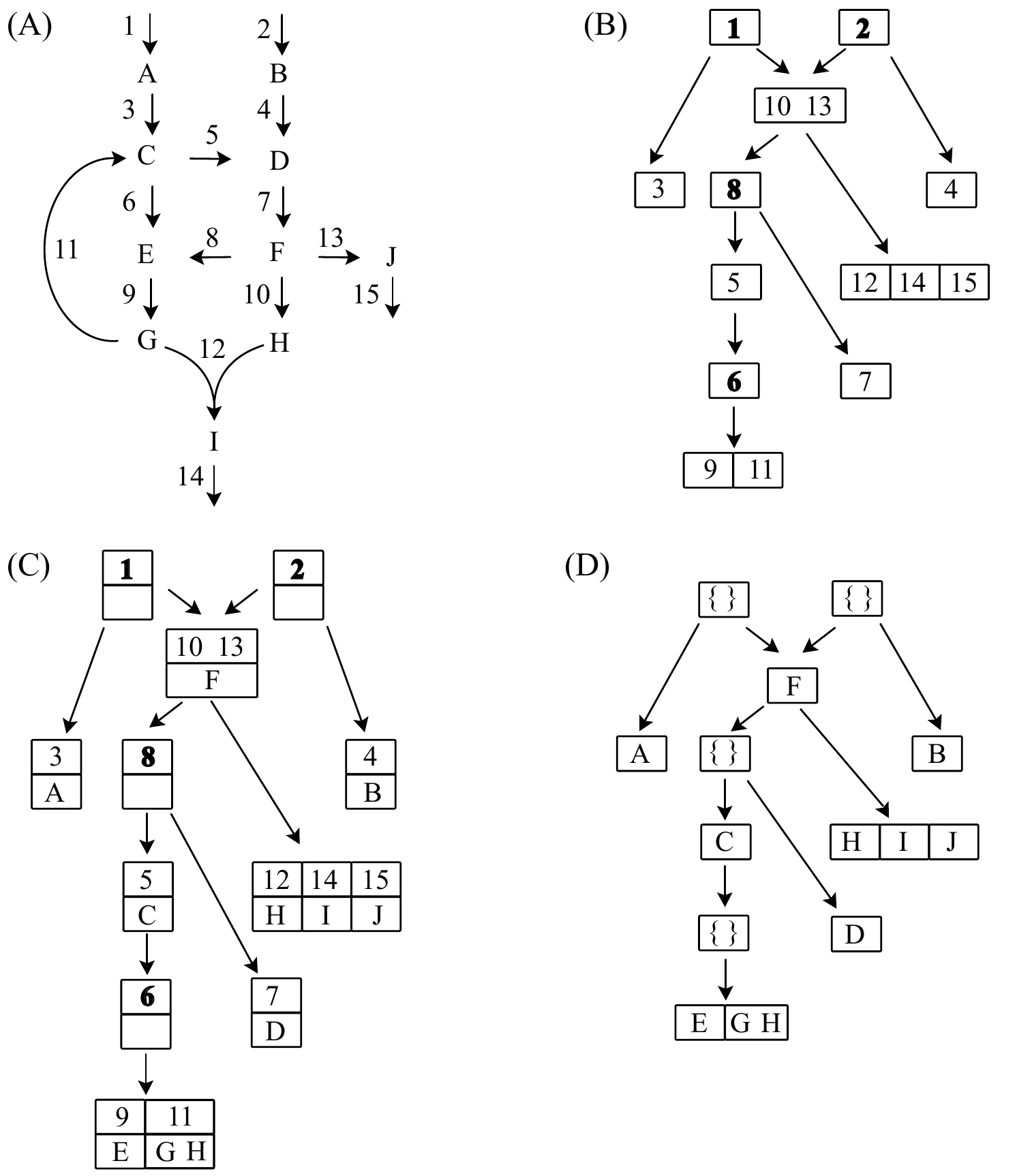}
\caption{\emph{
(A) Hypothetical network with 15 reactions, $\mathcal{E} = \{1,\ldots,15\}$ and 10 metabolites $\mathcal{M} = \{\texttt{A}, \ldots, \texttt{J}\}$. The only bimolecular reaction is $j=12$; see \cite{fiedler15example}.
(B) The pure flux influence graph $F$.
Equivalence classes $\varphi^*$ of mutual influence are denoted by single boxes. The only box with more than one element is $\varphi^* = \langle 10,13 \rangle$.
Self influences $1, 2, 6, 8$ are marked in boldface.
Arrows to arrays with several columns of boxes are pointing to each column, separately.
(C) The full flux influence graph $\mathcal{F}$ with direct metabolite influence set $\mathcal{M}^d(\varphi^*)$ attached below each vertex box $\varphi^*$ of $F$.
Omitting the flux classes $\varphi^*$ from $\mathcal{F}$, we obtain the graph $(D)$ of all possible metabolite influence sets $I_{\mathcal{M}}(\varphi^*)$.
These are given by the union along all directed paths emanating from any given vertex $\mathcal{M}^d (\varphi^*)$. The starting vertex is included and may be an empty set vertex $\{ {~}\}$.
}}
\label{fig:3.1}
\end{figure}

\begin{figure}[]
\centering \includegraphics[width=\textwidth]{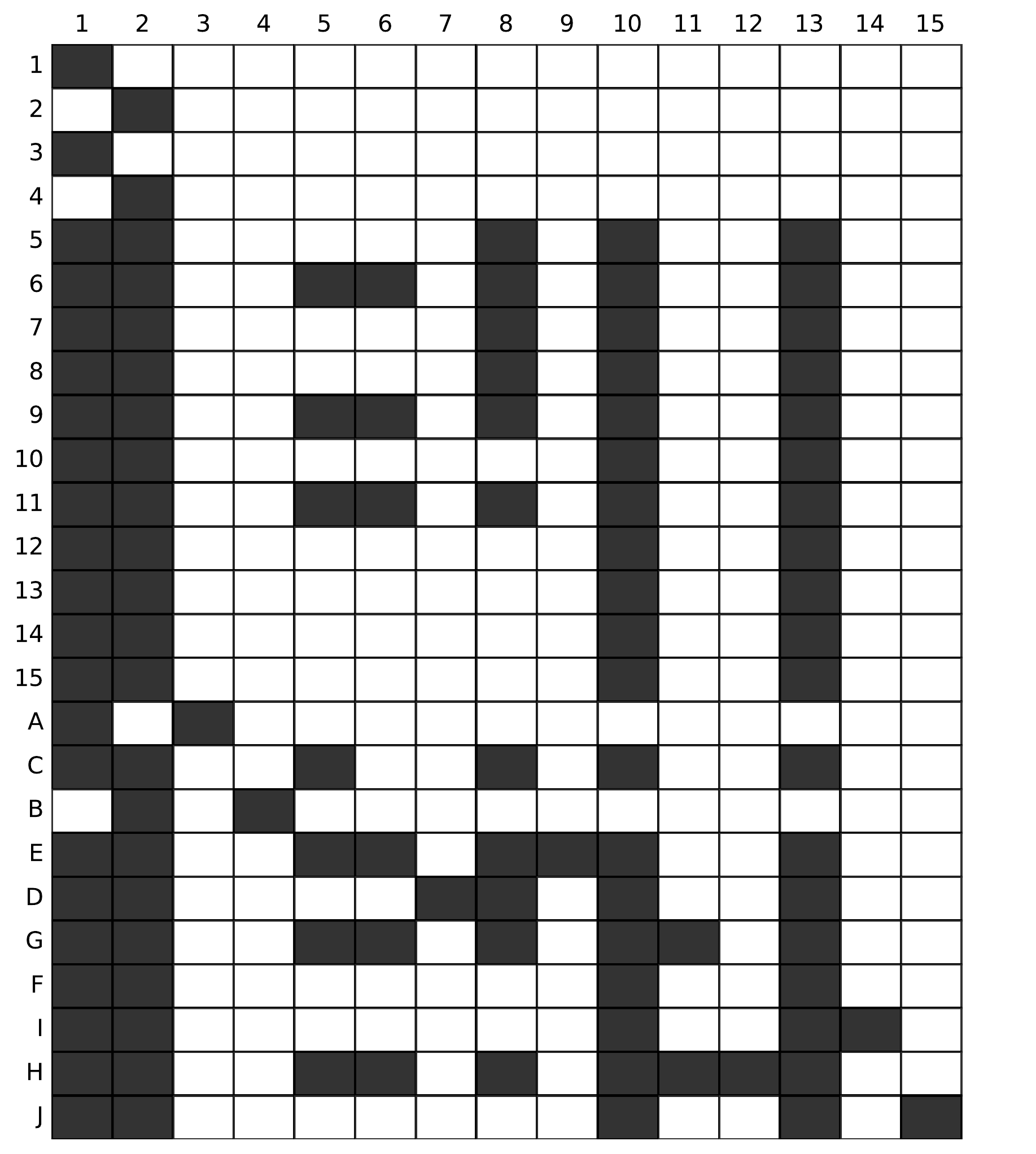}
\caption{\emph{
Algebraically nonzero entries (black) of the sensitivity matrix $B_{\beta \alpha}^{-1}$, of example network fig.~\ref{fig:3.1}(a).
A black area in row $\beta \in \mathcal{E} \cup \mathcal{M}$ and column $\alpha \in \mathcal{E}$ indicates nonzero influence $\alpha \leadsto \beta$ of the rate of reaction $\alpha$ on the flux of reaction $\beta$, for $\beta \in \mathcal{E}$, and on the steady state concentration of metabolite $\beta$, for $\beta \in \mathcal{M}$.
}}
\label{tbl:3.1}
\end{figure}

The $25 \times 25$ matrix $B$ of \eqref{eq:3.1}, \eqref{eq:3.2} can easily be inverted symbolically.
In our example,
	\begin{equation}
	\det B =
	- r_3\, r_4\, r_5\, r_7\, r_9 \,
	r_{11}\, r_{12H}\,r_{14}\,r_{15}\,(2r_{10} +
	r_{13})\,.
	\label{eq:3.12}
	\end{equation}
We have omitted the redundant input metabolite index $m$ in $r_{jm}$, for monomolecular reactions $j$.
In \eqref{eq:3.12} we see explicitly what it means for  $\det B$ to be nonzero, algebraically.
We have to require nonzero prefactors $r_3 ,\dots ,r_{15}$ and the linear nondegeneracy $2r_{10} + r_{13} \neq 0$.
In this explicit sense, the bimolecular chemical reaction network of fig.~\ref{fig:3.1}(a) is regular, algebraically; see also theorem~\ref{thm:2.1}.

The feed reactions are $\{1,2\}$ and the exit reactions are $\{14,15\}$.
We have omitted the formal metabolite entry $0$ in these open system reactions.
The seven single children $j$, as defined just after theorem \ref{thm:2.2}, are
	\begin{equation}
	j \in \{ 3,4,7,9,12,14,15\}\,.
	\label{eq:3.11}
	\end{equation}
By theorem~\ref{thm:2.2} they have no flux influence, and therefore define terminal sinks of the pure flux influence graph $F$; see fig.~\ref{fig:3.1}(b). The remaining terminal sink $j = 11$ is not a single child.

Let us study \eqref{eq:3.12} in a little more detail, in terms of child selections $J$ and \eqref{eq:2.4b}--\eqref{eq:2.4f}.
The seven single children $j$ in \eqref{eq:3.11} are forced to appear in the index set of any monomial in \eqref{eq:3.12}.
Indeed, their unique respective single-child mothers $m \vdash j$ have no choice but to select their only child $j = J(m)$; see \eqref{eq:2.3} and \eqref{eq:2.4f}.
Since the bimolecular reaction $j =12$ is the single child of metabolite \texttt{H}, this also forces the index $j=J(\texttt{G}) =11$ of $r_{11}$ to appear in the prefactor of \eqref{eq:3.12}.
The only remaining choices for a child selection $J$ are
	\begin{equation}
	J(\texttt{C}) \in \{ 5,6\} \quad \text{and} \quad 
	J(\texttt{F}) \in \{8,10,13\}\,.
	\label{eq:3.13}
	\end{equation}
The choice $J(\texttt{C}) = 6$ immediately leads to the indicator $1_{\{6,9,11\}}$ of the network cycle $\texttt{C} \overset{6}\rightarrow \texttt{E} \overset{9}\rightarrow \texttt{G}\overset{11}\rightarrow \texttt{C}$ as a kernel vector of $S$ supported on $J(\mathcal{M})$.
Indeed $J(\texttt{E}) =9$ is a single child, and we just saw why $J(\texttt{G}) =11$ is also forced to hold.
Therefore we must choose $J(\texttt{C}) =5$ in theorem \ref{thm:2.1}.
This explains why $r_5$ must also appear in the prefactor of \eqref{eq:3.12}.

We could easily proceed with such elementary arguments to fully derive \eqref{eq:3.12} by hand.
Likewise, the modified child selections of theorems~\ref{thm:2.2} and~\ref{thm:2.3} would determine all flux and metabolite influences of rate perturbations $j^*$, explicitly.
Instead, we present a symbolic version of $B^{-1}$ in figure~\ref{tbl:3.1}.
A black area entry for $B_{\beta \alpha}^{-1}$ in row $\beta \in \mathcal{E} \cup \mathcal{M}$ and column $\alpha  \in \mathcal{E}$ of $B^{-1}$ is equivalent to an algebraically nonzero component $z_\beta^\alpha$ of the response, i.e. to the influence $\alpha \leadsto \beta$; see \eqref{eq:3.1}, \eqref{eq:3.3}.
By the construction \eqref{eq:2.26}--\eqref{eq:2.27b}, \eqref{eq:2.29} this immediately defines the flux influence di-graphs $F, \mathcal{F}$ of fig.~\ref{fig:3.1}(b),(c).
Theorem~\ref{thm:2.7} explains how these graphs determine all flux influence sets $I_{\mathcal{E}}(j^*)$ and metabolite influence sets $I_{\mathcal{M}}(j^*)$, explicitly, by following directed paths downward in the diagrams.
Note the mother-child pairs $m^* \vdash j^*$ which provide terminal sink vertices $j^* \{ m^*\}$ in $\mathcal{F}$, fig.~\ref{fig:3.1}(c).
The bimolecular input $\{\texttt{G,H}\}$ of reaction $12$, interestingly, cannot respond to a perturbation of $j^* =12$, because $j^* =12$ is the single child of $m^* =\texttt{H}$.
In fact $\{\texttt{G,H}\}$  arises in response to the terminal sink $j^* =11$ of the di-graph $F$.
Indeed $j^* =11$ cannot influence itself, or any other flux, by theorem~\ref{thm:2.2}, because we already  saw that $j^* = 11= J(\texttt{G})$ must appear in any child selection.
Therefore $j^* \leadsto \texttt{G}$, in order not to change its own flux.
That metabolite change induces $j^* \leadsto \texttt{H}$ because the bimolecular flux of $j'=12$ is not influenced by $j^* =11$.
This explains the ``bimolecular'' sink vertex $11\{ \texttt{G,H}\}$ and the ``monomolecular'' sink vertex $12\{ \texttt{H}\}$ in the flux influence graph $\mathcal{F}$.

We have detailed in theorem~\ref{thm:2.7} how any metabolite influence set $I_{\mathcal{M}}(j^*) = I_{\mathcal{M}}(\varphi^*)$ arises as the union of the direct influence sets $\mathcal{M}^d(\varphi')$ along all downward directed paths from $\varphi^* \mathcal{M}^d(\varphi^*)$, including that starting vertex itself; see \eqref{eq:2.28a}, \eqref{eq:2.28b}.
To obtain the metabolite influence sets we can therefore simply omit the part $\varphi^*$ of the vertex labels in $\mathcal{F}$ and take unions of the remaining metabolite annotations $\mathcal{M}^d(\varphi')$ along downward directed paths; see fig.~\ref{fig:3.1}(d).
Note that some empty label vertices $\{ {~} \}$ arise, which cannot be omitted.
For example, the omission of $\langle 8 \rangle \{ {~} \}$ would claim that $\{\texttt{C,D,E,G,H}\} = I_{\mathcal{M}}(8)$ is not a metabolite influence set of any single flux perturbation.
Omission of $\langle 6 \rangle \{ {~} \}$, likewise, would claim that $\{\texttt{E,G,H}\} = I_{\mathcal{M}}(6)$ does not occur.

Our simple examples indicate a wealth of sensitivity information which is extracted from structural assumptions, alone.
Transitivity theorem~\ref{thm:2.4} allows us to display all flux influences in a single diagram.
And the sparse, and highly structured, inverse $B^{-1}$ of the sparse network/stoichiometry matrix $B$ testifies against conventional ``knowledge'' that inverses of sparse matrices are ``not sparse''.


\section{Proofs of theorems 2.1--2.3}
\label{sec4:proofs}

Our starting point is the matrix $B$ of section~\ref{sec3:theoex}; see \eqref{eq:3.1}--\eqref{eq:3.3}.
It is easy to block-diagonalize and to invert $B$ explicitly:
	\begin{equation}
	\begin{pmatrix}
	\text{id}_\mathcal{E} & 0\\
	S & \text{id}_\mathcal{M}
	\end{pmatrix}
	\begin{pmatrix}
	-\text{id}_\mathcal{E} & R\\
	S & 0
	\end{pmatrix}
	\begin{pmatrix}
	\text{id}_\mathcal{E} & R\\
	0 & \text{id}_\mathcal{M}
	\end{pmatrix}
	=
	\begin{pmatrix}
	-\text{id}_\mathcal{E} & 0\\
	0 & SR
	\end{pmatrix}\,.
	\label{eq:4.1}
	\end{equation}
With $E$:= $|\mathcal{E}|$ counting all reactions, this implies	
	\begin{equation}
	\det B = (-1)^E \det (SR)\,.
	\label{eq:4.2}
	\end{equation}
Hence $B$ is invertible, if and only if $SR$ is, with
	\begin{equation}
	B^{-1} =	
	\begin{pmatrix}
	-\text{id}_\mathcal{E} + R(SR)^{-1}S \quad&
	R(SR)^{-1}\\
	(SR)^{-1} S \quad& 
	(SR)^{-1}
	\end{pmatrix}\,.
	\label{eq:4.3}
	\end{equation}
Our proofs of theorems~\ref{thm:2.1}--\ref{thm:2.3} are all based on the Cauchy-Binet formula \cite{gantmacher} for determinants like $\det (SR)$ in \eqref{eq:4.2}.

Throughout this section, and for any matrix $B$, let
	\begin{equation}
	B_\varrho^\sigma
	\label{eq:4.4}
	\end{equation}
denote the submatrix of $B$ which consists of rows in the index set $\varrho$ and columns in the index set $\sigma$, only.
We frequently omit braces $\{ j\}$ for single element sets; for example $S_m$ and $S^j$ denote row $m$ and column $j$ of the stoichiometric matrix $S$, respectively.

\begin{proof}[{\textbf{Proof of theorem~2.1}}]
By the Cauchy-Binet formula, \cite{gantmacher}, and \eqref{eq:4.2},
	\begin{equation}
	\det (SR) =
	\sum_{\mathcal{E}' \in \mathcal{E}_M}
	\det S^{\mathcal{E}'} \cdot 
	\det R_{\mathcal{E}'}\,.
	\label{eq:4.5}
	\end{equation}
The sum runs over the set $\mathcal{E}_M$ of all $\mathcal{E}' \subseteq \mathcal{E}$ with $|\mathcal{E}'| = M = | \mathcal{M}|$ elements.
More explicitly, with notation \eqref{eq:2.4d},
	\begin{equation}
	\det R_{\mathcal{E}'} =
	\sum_{J: 
	\mathcal{M} \rightarrow \mathcal{E}'}
	\text{sgn}\,J \cdot \textbf{r}^J
	\label{eq:4.6}
	\end{equation}
where $\text{sgn}\,J$ denotes the signature, or parity, of the bijection $J$: $\mathcal{M} \rightarrow \mathcal{E}'$.
The definition of $\text{sgn}\,J$ refers to our arbitrary, but fixed, labeling of metabolites $\mathcal{M} = \{1,\ldots , M\}$ and reactions $\mathcal{E} = \{1,\ldots , E\}$. The integer labels fix natural orderings on $\mathcal{M}$ and on $\mathcal{E}' := J(\mathcal{M}) \subseteq \mathcal{E}$, respectively. The orders define a natural identification of $\mathcal{M}$ with $\mathcal{E'}$, which allows us to view $J$ as a permutation of $\mathcal{M}$. By $\text{sgn}\,J = \pm 1$ we denote the signature of that permutation.

The monomial $\textbf{r}^J$ of \eqref{eq:2.4d} is nontrivial, if and only if $J$ is a child selection; see \eqref{eq:1.5}, \eqref{eq:1.6} and \eqref{eq:2.2}, \eqref{eq:2.3}.
This proves \eqref{eq:2.4c}, \eqref{eq:2.4e} with
	\begin{equation}
	a_J = \text{sgn}\,J \cdot \det S^{J(\mathcal{M})}\,,
	\label{eq:4.7}
	\end{equation}
and theorem~\ref{thm:2.1}.
\end{proof}

Henceforth we assume $\det SR \neq 0$, alias $\det B \neq 0$; see \eqref{eq:4.2}. Let $m_1, m_2 \in \mathcal{M}$.
By Cramer's rule and Cauchy-Binet we obtain
	\begin{equation}
	\begin{aligned}
	\det (SR) (SR)_{m_1m_2}^{-1} &=
	(-1)^{m_1+m_2} \det 
	(SR)_{\mathcal{M} \smallsetminus m_2}^{\mathcal{M}\smallsetminus m_1} =\\
	&= (-1)^{m_1+m_2} \det (S_{\mathcal{M} \smallsetminus m_2} R^{\mathcal{M} \smallsetminus m_1}) =\\
	&= (-1)^{m_1+m_2} \sum_{\mathcal{E}' \in \mathcal{E}_{M-1}} \det S _{\mathcal{M}\smallsetminus m_2}^{\mathcal{E}'}
	\cdot \det R_{\mathcal{E}'}^{\mathcal{M} \smallsetminus m_1}\,.
	\end{aligned}
	\label{eq:4.8}
	\end{equation}
By \eqref{eq:3.1}--\eqref{eq:3.3} we already know how $\alpha \leadsto \beta$ is equivalent to
	\begin{equation}
	0 \neq z_\beta^\alpha =
	-e_\beta^T B^{-1} e_\alpha =
	-B_{\beta \alpha}^{-1}\,,
	\label{eq:4.9}
	\end{equation}
algebraically.
By \eqref{eq:4.3}, \eqref{eq:4.5}, \eqref{eq:4.8} we already know how to evaluate such terms.

\begin{proof}[{\textbf{Proof of theorem~2.2(i)}}]
We have to consider the special case $\alpha = \beta = j^* \in \mathcal{E}$.
Without loss of generality, we may relabel reactions such that $j^* = 1$. 
Starting with \nopagebreak \eqref{eq:4.2},\eqref{eq:4.9}, we obtain (not quite immediately)
\pagebreak
	\begin{equation}
	\begin{aligned}
	&(-1)^E z_{j^*}^{j^*} \det B =\\
	&\quad =
	-\det (SR) B_{j^*j^*}^{-1}=
	\det (SR) - \big( R \det 
	(SR) (SR)^{-1}S\big)_{j^*j^*} =\\
	&\quad =
	\det (SR) - \sum_{m_1, m_2 \in \mathcal{M}}
	R_{j^*}^{m_1} \det (SR)(SR)_{m_1m_2}^{-1}
	S_{m_2}^{j^*} =\\
	&\quad =
	\det (SR)\, - \\
	&\quad -\sum_{m_1,m_2} 
	\sum_{\mathcal{E}' \in \mathcal{E}_{M-1}}
	(-1)^{m_2+j^*} S_{m_2}^{j^*} \det
	S_{\mathcal{M} \smallsetminus
	m_2}^{\mathcal{E}'} 
	\cdot (-1)^{m_1+j^*}
	R_{j^*}^{m_1} \det 
	R_{\mathcal{E}'}^{\mathcal{M} \smallsetminus
	m_1} =\\
	&\quad =
	\det	 (SR) - \sum_{\mathcal{E}'\in
	\mathcal{E}_{M-1}} \det \left(S^{j^*},
	S^{\mathcal{E}'}\right) \cdot \det
	\begin{pmatrix}
	R_{j^*} \\ R_{\mathcal{E}'}
	\end{pmatrix}
	=\\
	&\quad = \sum_{\mathcal{E}'' \in \mathcal{E}_M} 
	\det S^{\mathcal{E}''} \cdot \det R_{\mathcal{E}''} - \sum_{j^*
	\not\in \mathcal{E}' \in \mathcal{E}_{M-1}}
	\det \left(S^{j^*},
	S^{\mathcal{E'}}\right) \cdot \det
	\begin{pmatrix}
	R_{j^*} \\ R_{\mathcal{E}'}
	\end{pmatrix}
	=\\
	&\quad = \sum_{j^* \not\in \mathcal{E}'' \in
	\mathcal{E}_M} 
	\det S^{\mathcal{E}''} \cdot \det R_{\mathcal{E}''}\,.
	\end{aligned}
	\label{eq:4.10}
	\end{equation}
We have used expansions of determinants with respect to a prepended column $S^{j^*}$ or row $R_{j^*}$.
We also omitted the cases $j^* \in \mathcal{E}'$ of duplicate rows and columns.

The proof of theorem~\ref{thm:2.2}(i) then concludes analogously to \eqref{eq:4.6}, \eqref{eq:4.7}, and shows
	\begin{equation}
	z_{j^*}^{j^*} \det (SR) =
	\sum_{J: \mathcal{M} \rightarrow
	\mathcal{E} \smallsetminus j^*}
	a_J\, \mathbf{r}^J\,,
	\label{eq:4.11}
	\end{equation}	
where $J$ are child selections and
	\begin{equation}
	a_J= \text{sgn}\,J \cdot 
	\det S^{J(\mathcal{M})}\,.
	\label{eq:4.12}
	\end{equation}	 
\end{proof}

\begin{proof}[{\textbf{Proof of theorem~2.2(ii)}}]
This time we have to consider $\alpha = j^* =1$, $\beta = j' =2$, without loss of generality.
We proceed along the lines of \eqref{eq:4.10}--\eqref{eq:4.12} to prove
	\begin{equation}
	\begin{aligned}
	&(-1)^{E-1} z_{j'}^{j^*} \det B =\\
	&\quad =
	\det (SR) B_{j' j^*}^{-1}
	=\sum_{m_1,m_2 \in \mathcal{M}} R_{j'}^{m_1}
	\det (SR) (SR)_{m_1m_2}^{-1} S_{m_2}^{j^*}=\\
	&\quad =\sum_{m_1,m_2} \sum_{\mathcal{E}'\in
	\mathcal{E}_{M-1}} (-1)^{m_2+j^*}
	S_{m_2}^{j^*}\det 
	S_{\mathcal{M}\smallsetminus m_2}^
	{\mathcal{E}'}
	 \cdot(-1)^{m_1 +j^*}
	R_{j'}^{m_1} \det
	R_{\mathcal{E}'}^{\mathcal{M}
	\smallsetminus m_1}=\\
%
	&\quad =\sum_{\mathcal{E}' \in\mathcal{E}_{M-1}}
	\det \left( S^{j^*}, S^{\mathcal{E}'}
	\right) \cdot \det
	\begin{pmatrix}
	R_{j'} \\ R_{\mathcal{E}'}
	\end{pmatrix}
	=\\
	&\quad =\sum_{j^*,\, j' \not\in \mathcal{E}' \in
	\mathcal{E}_{M-1}} \det\left( S^{j^*},
	S^{\mathcal{E}'}
	\right) \cdot \det
	\begin{pmatrix}
	R_{j'} \\ R_{\mathcal{E}'}
	\end{pmatrix}
	=\\
	&\quad = \sum_{j^*,\, j' \not\in \mathcal{E}' \in
	\mathcal{E}_{M-1}} \det S^{\mathcal{E}' \cup j^*} \cdot \det  R_{\mathcal{E}' \cup j'}\,.
	\end{aligned}
	\label{eq:4.13}
	\end{equation}
Again we have expanded determinants with respect to a prepended column $S^{j^*}$ or row $R_{j'}$, and we have safely omitted the zero determinants caused by duplicate columns $S^{j^*}$ or rows $R_{j'}$.
This shows
	\begin{equation}
	z_{j'}^{j^*} \det (SR) =
	\sum_{j^* \not\in J(\mathcal{M}) \ni j'}
	a_J \, \mathbf{r}^J\,.
	\label{eq:4.14}
	\end{equation}
Here $J$ are child selections, and
	\begin{equation}
	a_J = -\;\text{sgn}J \cdot\det S^{J(\mathcal{M})_{\text{sw}}}
	\label{eq:4.15}
	\end{equation}
with the swapped columns
	\begin{equation}
	J(\mathcal{M})_{\text{sw}}
	\ := \ j^* \cup
	J(\mathcal{M}) \smallsetminus j'\,.
	\label{eq:4.16}
	\end{equation}
This completes the proof of theorem~\ref{thm:2.2}.
\end{proof}

\begin{proof}[{\textbf{Proof of theorem~2.3}}]
Quite similarly to the previous cases we only have to consider $\alpha = j^* = 1$.
For $\beta = E+m' = E +1$ we are also allowed to pick the first element $m'=1$
of $\mathcal{M}$.
We proceed as usual:
	\begin{equation}
	\begin{aligned}
	(-1)^{E-1} z_{E+m'}^{j^*} \det B &=
	\det (SR) B_{E+m' ,\, j^*}^{-1}
	= \sum_{m \in \mathcal{M}} \det
	(SR)(SR)_{m'm}^{-1} S_m^{j^*} = \\
	&= \sum_{m \in \mathcal{M}}
	\sum_{\mathcal{E}' \in \mathcal{E}_{M-1}}
	(-1)^{m+j^*} S_m^{j^*} \det
	S_{\mathcal{M} \smallsetminus
	m}^{\mathcal{E}'} \cdot \det
	R_{\mathcal{E}'}^{\mathcal{M}\smallsetminus
	 m'} =\\
	&= \sum_{\mathcal{E}'\in \mathcal{E}_{M-1}}
	\det \left( S^{j^*}, S^{\mathcal{E}'}
	\right) \cdot \det R_{\mathcal{E}'}
	^{\mathcal{M}\smallsetminus m'}=\\
	&=\sum_{j^* \not\in \mathcal{E}' \in
	\mathcal{E}_{M-1}}
	\det S^{\mathcal{E}' \cup j^*} \cdot \det R_{\mathcal{E}'}^
	{\mathcal{M} \smallsetminus m'}\,.
	\end{aligned}
	\label{eq:4.17}
	\end{equation}
Note how we have substituted $j^*$ for $m'=1=j^*$ in \eqref{eq:4.17}. This shows
	\begin{equation}
	z_{E+m'}^{j^*} \det (SR) =
	\sum_{J^\vee} a_{J^\vee}\,
	\mathbf{r}^{J^\vee}
	\label{eq:4.18}
	\end{equation}
where $J^\vee$: $\mathcal{M} \smallsetminus m' \longrightarrow \mathcal{E} \smallsetminus j^*$ is a partial child selection.
We extend $J^\vee$ to a bijection
	\begin{equation}
	J: \quad \mathcal{M} \ \longrightarrow \  J(\mathcal{M}):=
	 J^\vee (\mathcal{M} \smallsetminus m') 
	 \cup j^*\,,	
	 \label{eq:4.19}
	\end{equation}
defining $J(m')$:= $j^*$.
Here $J$ need not be a child selection.
We obtain a nonzero coefficient
	\begin{equation}
	\begin{aligned}
	a_J &= -\;\text{sgn}\,J^\vee 
	\cdot\det 
	S^{J^\vee(\mathcal{M}\smallsetminus m')
	\cup j^*} =\\
	&= -\;\text{sgn}\,J \cdot \det
	S^{J(\mathcal{M})}\,.
	\end{aligned}
	\label{eq:4.20}
	\end{equation}
This completes the proof of theorem~\ref{thm:2.3}.
\end{proof}


\section{Augmenticity}
\label{secB4:aug}

As a corollary of theorems~\ref{thm:2.1}--\ref{thm:2.3} we study how the influence relation $j^* \leadsto \alpha \in \mathcal{E} \cup \mathcal{M}$ is affected when we enlarge the network.
In theorem~\ref{thm:B4.1} below we observe how existing influences $j^* \leadsto \alpha$ within the smaller network $(\mathcal{E}_0, \mathcal{M}_0)$ persist in the larger, augmented network $ (\mathcal{E}_1, \mathcal{M}_1) \supseteq (\mathcal{E}_0, \mathcal{M}_0)$, possibly enriched by new, additional influences.
To distinguish this ``monotonicity'' feature from other, more mundane and elementary features involving monotone reaction rates or comparison type theorems in a single network, we use the term \emph{augmenticity} for such ``monotonicity'' under augmentation of networks.

To be more precise we call a network $(\mathcal{E}_1, \mathcal{M}_1)$ an \emph{augmentation} of a network $(\mathcal{E}_0, \mathcal{M}_0)$, in symbols
	\begin{equation}
	(\mathcal{E}_1, \mathcal{M}_1) \supseteq  (\mathcal{E}_0, \mathcal{M}_0)\,, 
	\label{eq:B4.0}
	\end{equation}
if $\mathcal{E}_1 \supseteq  \mathcal{E}_0,\ \mathcal{M}_1 \supseteq  \mathcal{M}_0$, and the stoichiometric vectors $y^j, \bar{y}^j$ of the networks coincide for all $j \in \mathcal{E}_0$; see \eqref{eq:1.1}.
As always we have identified $y^j, \bar{y}^j \in \mathcal{M}_0 \subseteq \mathcal{M}_1$ by zero padding; see notation \eqref{eq:1.2}.
In particular, the associated stoichiometric matrices $S_0, S_1$ satisfy
	\begin{equation}
	S_0 = S^{\mathcal{E}_0}_{1,\, \mathcal{M}_0} 
	\quad \text{and}\quad 
	0= S_{1, \, \mathcal{M}_1 \smallsetminus \mathcal{M}_0}^{\mathcal{E}_0}\,.
	\label{eq:B4.1}
	\end{equation}
We also call $(\mathcal{E}_0, \mathcal{M}_0)$ a \emph{subnetwork} of $(\mathcal{E}_1, \mathcal{M}_1)$.

Admittedly, new reactions or metabolites may drastically affect the numerical values (and even the very existence and multiplicity) of existing steady states.
Our viewpoint of qualitative sensitivity, however, is only concerned with the collection of algebraically nontrivial response patterns, as derived from the stoichiometric vectors $y^j, \bar{y}^j$.
Therefore the following augmenticity theorem is surprisingly simple.

\begin{thm}
Assume the network $(\mathcal{E}_0, \mathcal{M}_0)$ is regular, algebraically, as in theorem~\ref{thm:2.1}.
Let the network $(\mathcal{E}_1, \mathcal{M}_1)$ be an augmentation of the subnetwork $(\mathcal{E}_0, \mathcal{M}_0) \subseteq  (\mathcal{E}_1, \mathcal{M}_1)$.
Assume there exists a partial child selection
	\begin{equation}
	J^\vee: \quad
	\mathcal{M}_1 \smallsetminus \mathcal{M}_0 \ \longrightarrow \ 
	\mathcal{E}_1 \smallsetminus \mathcal{E}_0
	\label{eq:B4.2}
	\end{equation}
such that the associated restriction of the stoichiometric matrix $S_1$ of the augmented network $(\mathcal{E}_1, \mathcal{M}_1)$ is nonsingular:
	\begin{equation}
	\det S_{{\mathcal{M}_1} \smallsetminus {\mathcal{M}_0}}
	^{J^{\vee} (\mathcal{M}_1 \smallsetminus \mathcal{M}_0)}
	\quad \neq \quad 0\,.
	\label{eq:B4.3}
	\end{equation}
Then the augmented network $(\mathcal{E}_1, \mathcal{M}_1)$ is also regular, algebraically.
Moreover, any influence
	\begin{equation}
	\mathcal{E}_0 \ni j^* \,\, \leadsto \,\, 
	\alpha \in \mathcal{E}_0 \cup \mathcal{M}_0
	\label{eq:B4.4}
	\end{equation}
within the subnetwork $(\mathcal{E}_0, \mathcal{M}_0)$ remains an influence in the augmented network $(\mathcal{E}_1, \mathcal{M}_1)$.
\label{thm:B4.1}
\end{thm}

\begin{proof}[{\textbf{Proof.}}]
To prove algebraic regularity of the augmented network, we will invoke theorem~\ref{thm:2.1}. Since the subnetwork $(\mathcal{E}_0, \mathcal{M}_0)$ is assumed to be algebraically regular, there exists a child selection $J_0$: $\mathcal{M}_0 \rightarrow \mathcal{E}_0$ such that
	\begin{equation}
	\det S_0^{J_0 (\mathcal{M}_0)}\quad \neq \quad  0\,.
	\label{eq:B4.5}
	\end{equation}
Let us extend $J_0$ by the partial child selection $J^\vee$ to a map
	\begin{equation}
	J_1(m):=
	\left\lbrace 
	\begin{aligned}
	&J_0 (m) \in \mathcal{E}_0\,, 
	&\quad &\text{for} \quad m \in \mathcal{M}_0\,;\\
	&J^\vee (m) \in \mathcal{E}_1 \smallsetminus \mathcal{E}_0\,,
	&\quad &\text{for} \quad m \in \mathcal{M}_1 \smallsetminus \mathcal{M}_0\,.	
	\end{aligned}
	\right.
	\label{eq:B4.6}
	\end{equation}
Then $J_1$ is a child selection in the augmented network $(\mathcal{E}_1, \mathcal{M}_1)$.
We claim
	\begin{equation}
	\det S_1^{J_1 (\mathcal{M}_1)}\quad \neq \quad  0\,.
	\label{eq:B4.7}
	\end{equation}
Indeed the square restriction of the stoichiometric $S_1$ in \eqref{eq:B4.7} is block triangular, by extension property \eqref{eq:B4.1}.
By construction \eqref{eq:B4.6} the two diagonal blocks are associated to $J_0$ and $J^\vee$, respectively. 
Their determinants are nonzero by \eqref{eq:B4.5} and \eqref{eq:B4.3}, respectively. 
This establishes claim \eqref{eq:B4.7}. 
Invoking theorem~\ref{thm:2.1} proves algebraic regularity of the augmented network $(\mathcal{E}_1, \mathcal{M}_1)$.

Next assume influence $j^* \leadsto \alpha$ in the subnetwork $(\mathcal{E}_0, \mathcal{M}_0)$; see \eqref{eq:B4.4}.
Choose the associated (partial) child selections $J_0, J_0^\vee$ in $(\mathcal{E}_0, \mathcal{M}_0)$ as specified in theorems~\ref{thm:2.2}, ~\ref{thm:2.3}. Of course, the (partial) child selections $J_0, J_0^\vee$ depend on $j^*$ and on the choice of $\alpha \in \mathcal{E}_0 \cup \mathcal{M}_0$.
The same augmentation \eqref{eq:B4.6}, as before, then proves $j^* \leadsto \alpha$ is inherited by the augmented network $(\mathcal{E}_1, \mathcal{M}_1)$.
This proves the theorem.
\end{proof}

We  comment on the special case $\mathcal{M}_1 = \mathcal{M}_0$ of the above theorem, where the augmented network only adds some new reactions to the same set of metabolites.
Then a partial child selection $J^\vee$ is not required in \eqref{eq:B4.2}, \eqref{eq:B4.6} and all influences $j^* \leadsto \alpha$ in the subnetwork $(\mathcal{E}_0, \mathcal{M}_0)$  persist under the augmentation $(\mathcal{E}_1, \mathcal{M}_1) = (\mathcal{E}_1, \mathcal{M}_0) \supseteq (\mathcal{E}_0, \mathcal{M}_0)$.
Adding feed or exit reactions are particular examples: see section \ref{sec:disc}.
Once again, we caution our reader that such modifications actually may disrupt steady state analysis, even though our sensitivity results remain valid -- on a voided example.


\section{Transitivity}
\label{sec5:trans}

In this section we prove claims (i) and (ii) of transitivity theorem~\ref{thm:2.4}.
Although theorems~\ref{thm:2.2} and~\ref{thm:2.3} characterize flux influence $j^* \leadsto j'$ and metabolite influence $j^* \leadsto m'$ in complete detail, we did not succeed to prove our transitivity claims as a direct consequence of these characterizations.
The main obstacle was to relate, match, and merge the first child selection, given by the assumptions $\alpha \leadsto m$, or $ \alpha \leadsto j$, with the second child selection, given by $j \leadsto \beta$.
Instead, we will use differentiation with respect to an intermediary reaction term $r_{jm}$, where $m \vdash j$.

We first show how claim (i) of transitivity theorem~\ref{thm:2.4} implies claim (ii).
For $\alpha = j$ or $ j = \beta$, in assumption \eqref{eq:2.9}, there is nothing to prove.
Next, consider $\alpha \neq j \neq \beta$ and assume $\alpha = j^* \leadsto j \neq \alpha$. By definition \eqref{eq:1.16} of flux sensitivity,
	\begin{equation}
	0 \neq \Phi_j^\alpha := 
	(R \delta x^\alpha)_j=
	\sum_{m\, \vdash j} r_{jm}\, \delta x_m^\alpha
	\label{eq:5.1a}
	\end{equation}
must hold, algebraically.
This means that there exists at least one input $m \vdash j$ for which $\delta x_m^\alpha \neq 0$ holds, algebraically.
In other words
	\begin{equation}
	\alpha \leadsto m \vdash j\,;
	\label{eq:5.1b}
	\end{equation}
see \eqref{eq:3.3}.
Since we have also assumed $j \leadsto \beta$ in \eqref{eq:2.9}, assumption \eqref{eq:2.8} of theorem \ref{thm:2.4}(i) is satisfied.
This proves $\alpha \leadsto \beta$, as claimed in (ii).

It remains to prove claim (i).
In the notation of \eqref{eq:4.9}, assumption \eqref{eq:2.8} implies that
	\begin{equation}
	z_m^\alpha =
	-e_m^T B^{-1} e_\alpha \neq 0 \neq
	-e_\beta^T B^{-1} e_j  = z_\beta^j
	\label{eq:5.2}
	\end{equation}
both hold, algebraically; see \eqref{eq:3.3} again.
We have to show that
	\begin{equation}
	z_\beta^\alpha = 
	-e_\beta^T B^{-1} e_\alpha  \neq 0
	\label{eq:5.3}
	\end{equation}
holds, algebraically.
By \eqref{eq:4.3} and Cramer's rule \eqref{eq:4.8}, the explicit algebraic expression for $z_\beta^\alpha$ is at most fractional linear in the variable $r_{jm}$.
To show \eqref{eq:5.3} it is therefore sufficient to partially differentiate the algebraic expression \eqref{eq:5.3} with respect to $r_{jm}$ and show
	\begin{equation}
	\partial_{r_{jm}}\, z_\beta^\alpha \neq 0
	\label{eq:5.4}
	\end{equation}
holds, algebraically.
Note that definition \eqref{eq:3.2} of $B$ implies $\partial_{r_{jm}}B= e_je_m^T$, for $m \vdash j$.
Therefore \eqref{eq:5.2} implies
	\begin{equation}
	\begin{array}{lccccr}
	\partial_{r_{jm}} \, z_\beta^\alpha 
	& = &
	- e_\beta^T \partial_{r_{jm}}\, B^{-1}e_\alpha
	& = &
	e_\beta^T B^{-1} ( \partial_{r_{jm}}\, B)
	 B^{-1} e_\alpha & =\\
	& =  &
	e_\beta^T B^{-1} e_je_m^T B^{-1} e_\alpha 
	& = &
	\left(e_\beta^T B^{-1} e_j \right) \cdot
	\left( e_m^T B^{-1} e_\alpha\right) & =\\
	& = &
	z_\beta^j \cdot z_m^\alpha
	& \neq & 0 \,. &
	\end{array}
	\label{eq:5.5}
	\end{equation}
This little calculation proves \eqref{eq:5.4}, and transitivity theorem~\ref{thm:2.4}.


\FloatBarrier
\section{Computational Aspects}\label{sect:algo}

We briefly discuss how to calculate the influence graph, in this section.
We are aware of three schools of thought, which provide algorithms for computing the influence graph. First, we can consider the non-influence question $j^*\not\leadsto \beta$, i.e.~$B^{-1}_{\beta j^*}\equiv 0$, as a question of \emph{polynomial identity testing}. Fast probabilistic algorithms for this problem are available. 
Second, we could consider $B$ as a layered mixed matrix and employ deterministic matroid algorithms, following \cite{murota2009matrices}, to obtain a provably correct result.
A third viewpoint is described in \cite{giordano2015computing}.
Their algorithms do not just compute generic influences $j^*\leadsto \beta$, but even compute sign $B^{-1}_{\beta j^*}$, under certain additional assumptions. However, runtime is exponential in $|\mathcal M|+|\mathcal E|$.
This is impractical, already, for applications of moderate size like the TCA-cycle discussed in Section \ref{sec7:extca}. We only report on the fast probabilistic approach here.

As we have seen, every entry of $B^{-1}$ is a rational function in the rate variables $r_{jm}$.
Although the numerators and denominators are of degree at most $M=|\mathcal M|$, they may contain a number of monomials which grows exponentially with $M$.
Symbolic representations of $B^{-1}$ should therefore be avoided, at all cost, even for networks of moderate size.
To check for $B^{-1}_{\beta j^*}\equiv 0$, probabilistically, we evaluate the matrix inverse for specific values of $r_{jm}$ which are chosen at random.
These values need not be related to any actual numerical values of $r_{jm}$ in any biological application or \emph{in silico} simulation.
Instead we use the following Schwartz-Zippel lemma.

\begin{lem} \cite{schwartz1980fast}\label{lem:schwartz-zippel}
Let $\mathbb{F}$ be a field, and $q: \mathbb{F}^N \to \mathbb{F}$ a nonzero polynomial of degree at most $M$, in $N$ variables. Let $T\subseteq \mathbb{F}$ by any finite test set.
Let $P(q=0\text{ in }T^N)$ denote the probability to obtain $q(x)=0$ for some random $N$-vector $x$, uniformly distributed in $T^N$. Then
\begin{equation}
P\left(q=0 \text{ in } T^N\right)\ \le \  M/|T|\,.
\label{eq:8.1}
\end{equation}
\end{lem}

The field $\mathbb{F}$ is chosen to computationally recognize exact zeros.
This excludes floating point arithmetic.
There are two obvious choices: $\mathbb{F}=\mathbb{Q}$, or the finite Galois fields $\mathbb{F}=\mathrm{GF}(p^n)$ for some prime $p$.
We choose to work in $\mathbb{F}=GF(p)=\mathbb{Z}_p$, for simplicity and speed. 
We choose a moderately sized random prime $p\in [k, 2k]$.
Already for $k=2^{127}$, this makes $|T|=p$ so ludicrously large that
the Schwartz-Zippel lemma \ref{lem:schwartz-zippel} practically excludes false zero results $q(x)=0$ for ``unlucky'' random choices of the components $r_{jm}$ of $x$.

A more subtle danger is caused by \emph{unlucky primes} $p$. These are primes $p$ which divide any of the numerators, or the denominators, of the symbolic inverse $B^{-1}$. Unlucky primes may produce false zeros, or false singularities.

We crudely estimate the number of unlucky primes, as follows. Let $\ell$ denote an upper bound of the greatest common divisor of all terms appearing in the numerator or in the denominator of a single entry of $B^{-1}$.
Then, any number bounded by $\ell$ can have at most $\log \ell / \log k$ different prime factors in the range $[k, 2k]$, out of the asymptotically ${k}/{\log k}$ existing primes in the same range.
Hence, if we choose the prime $p\in [k,2k]$ uniformly at random, we need $k\gg \log \ell$ in order to avoid a single false zero due to unlucky primes.
This requires $k\gg (|\mathcal M|+|\mathcal E|)^2\log \ell$ in order to avoid all possible false zero entries in $B^{-1}$, independently.

The greatest common divisor of a set of positive integers is bounded above by their minimum.
The Cramer determinants of $B$ are bounded by Hadamard's inequality, i.e.~the matrix norm, and we obtain an upper bound $\ell \le \left(1+\max_j|S_j|\right)^{|\mathcal E|}$.
Here $|S_j|=\sum_{m\in\mathcal M} |S_{mj}|$, and $+1$ accounts for the $-\mathrm{id}_{\mathcal E}$-part of $B$.
Hence, our extremely crude estimate requires the following lower bound on $k$:
\begin{equation}
k\gg |\mathcal E|(|\mathcal E|+|\mathcal M|)^2\log(1+\max_j|S_j|).
\label{eq:size-of-k}
\end{equation}
Again, the factor $(|\mathcal E|+|\mathcal M|)^2$ counts the entries of $B^{-1}$, independently.

A value of $k=2^{127}$ satisfies the crude requirement \eqref{eq:size-of-k}, for any metabolic network in databases like \cite{KEGG, BernhardD}. Practically, this eliminates the problem of unlucky primes $p$ and unlucky rate entries $r_{jm}$.
The computational overhead over floating point arithmetic turns out to be very moderate for primes of such order.

In summary, a single matrix inversion of $B$ with random rates $r_{jm} \mod p$, for a random prime $2^{127}<p<2^{128}$, is sufficient to compute the influence relation $\leadsto$ with an error probability far below the probability of  manufacturing defects in the hardware and cosmic ray interference.
This matrix inversion is a trivial task on semi-modern hardware and for realistic sizes $|\mathcal E|+|\mathcal M| \le 500$ of the metabolic network.
All remaining tasks for the construction of the full flux influence graph -- computing strong connected components and a transitive reduction -- are at least as fast as the matrix inversion.

Our computations were done in the Sage framework, \cite{sage}, which internally uses the fast library FFPack for linear algebra over finite fields,  \cite{fflas-ffpack}.
Compared to floating point arithmetic, the matrix inversion over $\mathbb{Z}_p$ incurred a runtime overhead of less than a factor four, for random primes up to order $p\approx k\approx  2^{500}$.
\texttt{Runtimes for the TCA cycle (48 reactions, 29 metabolites) were on the order of milliseconds on a standard laptop.
}


\FloatBarrier
\section{Example: The carbon metabolic TCA cycle}
\label{sec7:extca}

\begin{table}[t!]
{\scriptsize
\begin{minipage}[t]{0.45\linewidth}
    \begin{tabular}[t]{|>{\texttt\bgroup}c<{\egroup}|>{\texttt\bgroup}c<{\egroup} @{\quad$\to$\quad} >{\texttt\bgroup}c<{\egroup}|}
\hline    
Reaction & Inputs &  Outputs\\
\hline
1 &  Glucose  + PEP & G6P + PYR\\
2a &  G6P & F6P\\
2b &  F6P & G6P\\
3 &  F6P & F1,6P\\
4 &  F1,6P & G3P  + DHAP\\
5 &  DHAP & G3P\\
6 &  G3P & 3PG\\
7a &  3PG & PEP\\
7b &  PEP & 3PG\\
8a &  PEP & PYR\\
8b &  PYR & PEP\\
9 &  PYR & AcCoA  + CO2\\
10 &  G6P & 6PG\\
11 &  6PG & Ru5P +  CO2\\
12 &  Ru5P & X5P\\
13 &  Ru5P & R5P\\
14a &  X5P  + R5P & G3P + S7P\\
14b &  G3P +  S7P & X5P  + R5P\\
15a &  G3P + S7P & F6P + E4P\\
15b &  F6P +  E4P & G3P +  S7P\\
16a &  X5P  + E4P & F6P  + G3P\\
16b &  F6P +  G3P & X5P +  E4P\\
17 &  AcCoA  + OAA & CIT\\
18 &  CIT & ICT\\
19 &  ICT & 2-KG + CO2\\
20 &  2-KG & SUC + CO2\\
21 &  SUC & FUM\\
22 &  FUM & MAL\\
23a &  MAL & OAA\\
23b &  OAA & MAL\\
\hline
    \end{tabular}%
\end{minipage}
\begin{minipage}[t]{0.45\linewidth}
    \begin{tabular}[t]{|>{\texttt\bgroup}c<{\egroup}|>{\texttt\bgroup}c<{\egroup} @{\quad$\to$\quad} >{\texttt\bgroup}c<{\egroup}|}
\hline
Reaction & Inputs &  Outputs\\
\hline
24a &  PEP +  CO2 & OAA\\
24b &  OAA & PEP  + CO2\\
25 &  MAL & PYR + CO2\\
26 &  ICT & SUC + Glyoxylate\\
27 &  Glyoxylate + AcCoA & MAL\\
28 &  6PG & G3P + PYR\\
29 &  AcCoA & Acetate\\
30 &  PYR & Lactate\\
31 &  AcCoA & Ethanol\\
\hline
f1 &   & Glucose\\
\hline
d1 &     Lactate   &\\
d2 &     Ethanol   &\\
d3 &     Acetate   &\\
d4 &  R5P & \\
d5 &  OAA &\\ 
d6 &  CO2 & \\
\hline
dd1 &  G6P & \\
dd2 &  F6P & \\
dd3 &  E4P & \\
dd4 &  G3P& \\
dd5 &  3PG & \\
dd6 &  PEP & \\
dd7 &  PYR &\\
dd8 &  AcCoA &\\
dd9 &  2-KG &\\
\hline
X1 & Glucose &\\
\hline
N1 & S7P & S1,7P\\
N2 & S1,7P & E4P + DHAP\\
\hline
    \end{tabular}%
\end{minipage}}
\caption{\emph{Reactions in the TCAC metabolic network; see \cite{ishii, ishii2}.}}\label{table:tca-reactions}
\end{table}

Let us illustrate our analysis with a realistic class of examples.
The previous paper \cite{fiedler15example} discussed a variant of the tricarboxylic citric acid cycle (TCAC) in \emph{E. coli}.
Perturbation experiments by knockout of enzymes have been reported in \cite{ishii, ishii2}.
The relevant reactions are listed in Table \ref{table:tca-reactions}. Table \ref{table:subnets} defines five variants $A$--$E$ of this network which we will discuss.
For a graphical representation of the metabolic network, see fig.~\ref{fig:tca-net}. 

We will only discuss sensitivity of steady states here, for largely arbitrary reaction rate functions.
For interesting examples of bifurcations and oscillations, based on prescribed steady states and mass action kinetics with certain compatible, but randomized, choices of rate coefficients, see \cite{grofeu, groTCA}.

\begin{table}
\begin{tabular}{|c|c|p{7cm}|}
\hline
Variant& Included reactions &Comment\\
\hline
A & \texttt{1-31},\texttt{f1}, \texttt{d1}-\texttt{d6}& Reduced Network from \cite{ishii}, as discussed in \cite{fiedler15example}.\\
B & \texttt{1-31},\texttt{f1}, \texttt{d1}-\texttt{d6}, \texttt{dd1}-\texttt{dd9}& Network from \cite{ishii}, augmented by further exit reactions.\\
C & \texttt{1-31},\texttt{f1}, \texttt{d1}-\texttt{d6}, \texttt{dd1}-\texttt{dd9}, \texttt{X1}& Artificial network to discuss \texttt{Glucose} decay.\\
D & \texttt{1-31},\texttt{f1}, \texttt{d1}-\texttt{d6}, \texttt{N1}, \texttt{N2}& Network proposed in \cite{ishii2}, introducing the novel metabolite \texttt{S1,7P} and re\-actions \texttt{N1}, \texttt{N2}, with reduced exit reactions in the spirit of $A$.\\
E & \texttt{1-31},\texttt{f1}, \texttt{d1}-\texttt{d6}, \texttt{dd1}-\texttt{dd9}, \texttt{N1}, \texttt{N2}& The union of networks $B$ and $D$.\\ \hline
\end{tabular}
\caption{\emph{Variants of the TCAC metabolic network discussed in the text.}}\label{table:subnets}
\end{table}

\begin{figure}\centering
\includegraphics[width=0.9\textwidth]{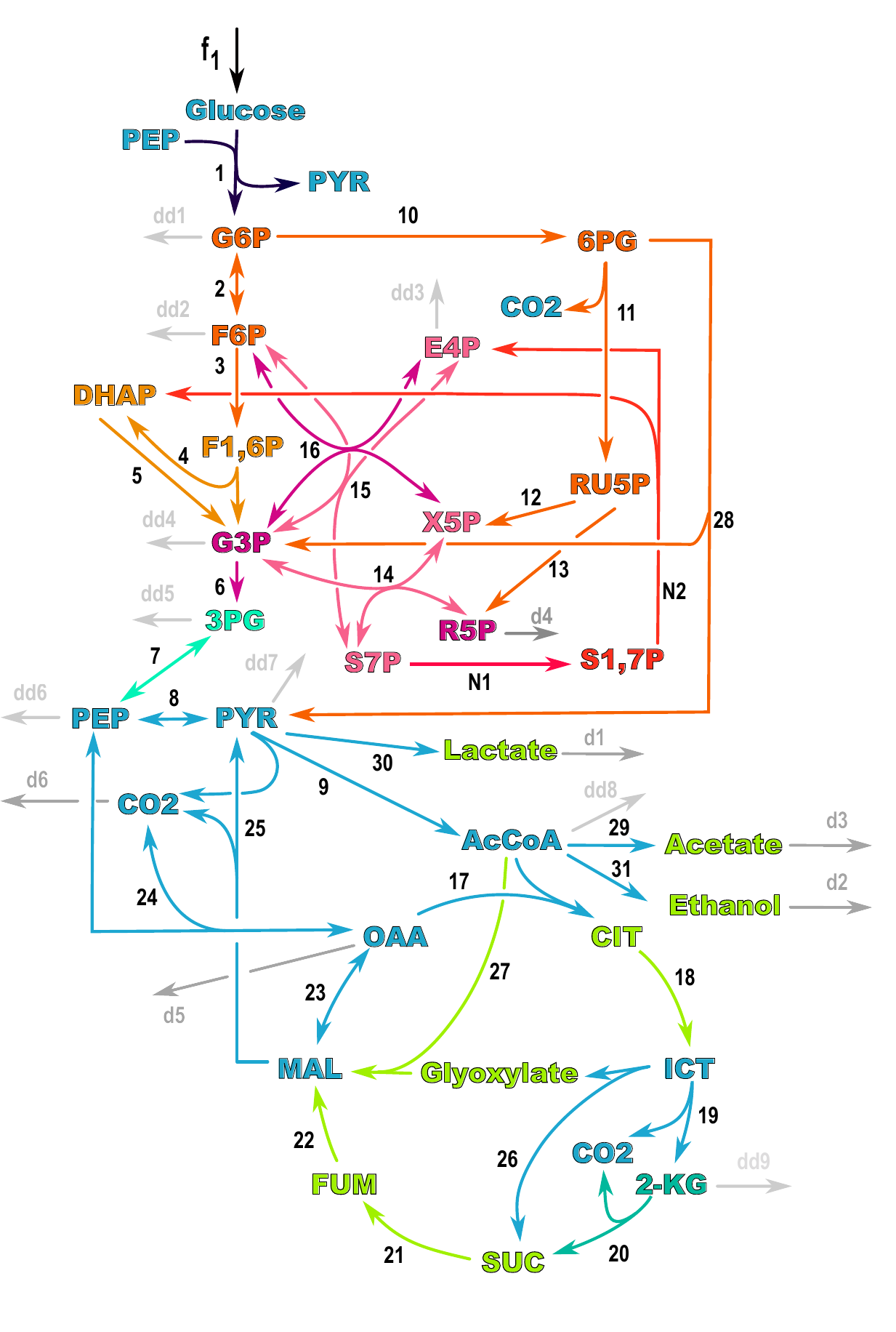}
\caption{\emph{The reaction network of the carbon metabolism TCA cycle of \emph{E.~coli}. The complete set of reactions is shown in Table \ref{table:tca-reactions}, and the discussed model variants are listed in Table \ref{table:subnets}.
Note that the metabolites \texttt{PEP}, \texttt{CO2}, \texttt{PYR} appear in multiple locations, in order to avoid excessive intersections of reaction edges.
Colors (online) indicate the grouping by influence in the models $A,D,E$ of \cite{ishii} and \cite{ishii2}; see fig.~\ref{fig:tca-ade}.
The graphical representations are courtesy of Anna~Karnauhova.}}\label{fig:tca-net}
\end{figure}

We will first discuss the model variant $A$, which consists of the internal reactions $1-31$, the feed reaction $\texttt{f1}$, and monomolecular exit reactions $\texttt{d1}-\texttt{d6}$.

The feed reaction \texttt{f1}, by definition, does not depend on any internal metabolite.
Indeed, the careful experiments in \cite{ishii, ishii2} normalized all measurements by total \texttt{Glucose} uptake.
This fixes \texttt{f1}, effectively.
It is not necessary to include the feed \texttt{f1} in our model, at all, as long as we are only interested in the influence graph of the remaining reactions. Adding reactions which have rates independent of all metabolites, in our model, will not change any influence relations between reactions and metabolites of the smaller model.
Indeed, see theorem \ref{thm:B4.1} with $\mathcal M_1=\mathcal M_0$.

Without the feed \texttt{f1}, on the other hand, the resulting network will not possess any nontrivial steady state at all: \texttt{Glucose} is consumed but never replenished, and the trivial zero state $x \equiv 0$ becomes globally attracting.
We can still compute, visualize and discuss the influence graph for such an incomplete network, formally, with the tacit understanding that it will become meaningful when we add the necessary feed reactions.
Evidently we do not even need to know, or account for, these external feed reactions, as long as we do not study their own influence on the network.

Exit reactions, in contrast, need to be included in the model. They are essential for the invertibility of $B$, and their presence or omission may affect the influence graph. The paper \cite{ishii} does not include the monomolecular exit reactions \texttt{d1}-\texttt{d3} explicitly.
Without them, however, the products \texttt{Lactate}, \texttt{Ethanol} and \texttt{Acetate} accumulate indefinitely.
Moreover, the resulting matrix $B$ becomes noninvertible.
The common practice of omitting such ``obvious'' reactions from networks, in the published literature, will be put under scrutiny below.
For a seriously cautioning menetekel see also section \ref{sec:disc}.

The full influence graph $\mathcal{F}$ of model $A$ is sketched in graphical form in fig.~\ref{fig:vanilla-tca}, and in tabular form in fig.~\ref{fig:tca-abc-table}.
As in fig.~\ref{fig:3.1}, we represent each vertex $\varphi \mathcal M^d(\varphi)$ of the full influence graph as a table with one column and two rows.
The top entry contains the reactions $j$ in the equivalence class $\varphi$ of mutual flux influence. The bottom entry contains the directly influenced metabolites $m \in \mathcal M^d(\varphi)$.
Self-influence is represented by boldface font, if the first row $\varphi$ has only a single entry. See section \ref{sec2:mainres}, theorem \ref{thm:2.7}, and in particular \eqref{eq:2.10} -- \eqref{eq:2.29}.

\begin{figure}\centering
\includegraphics[width=\textwidth]{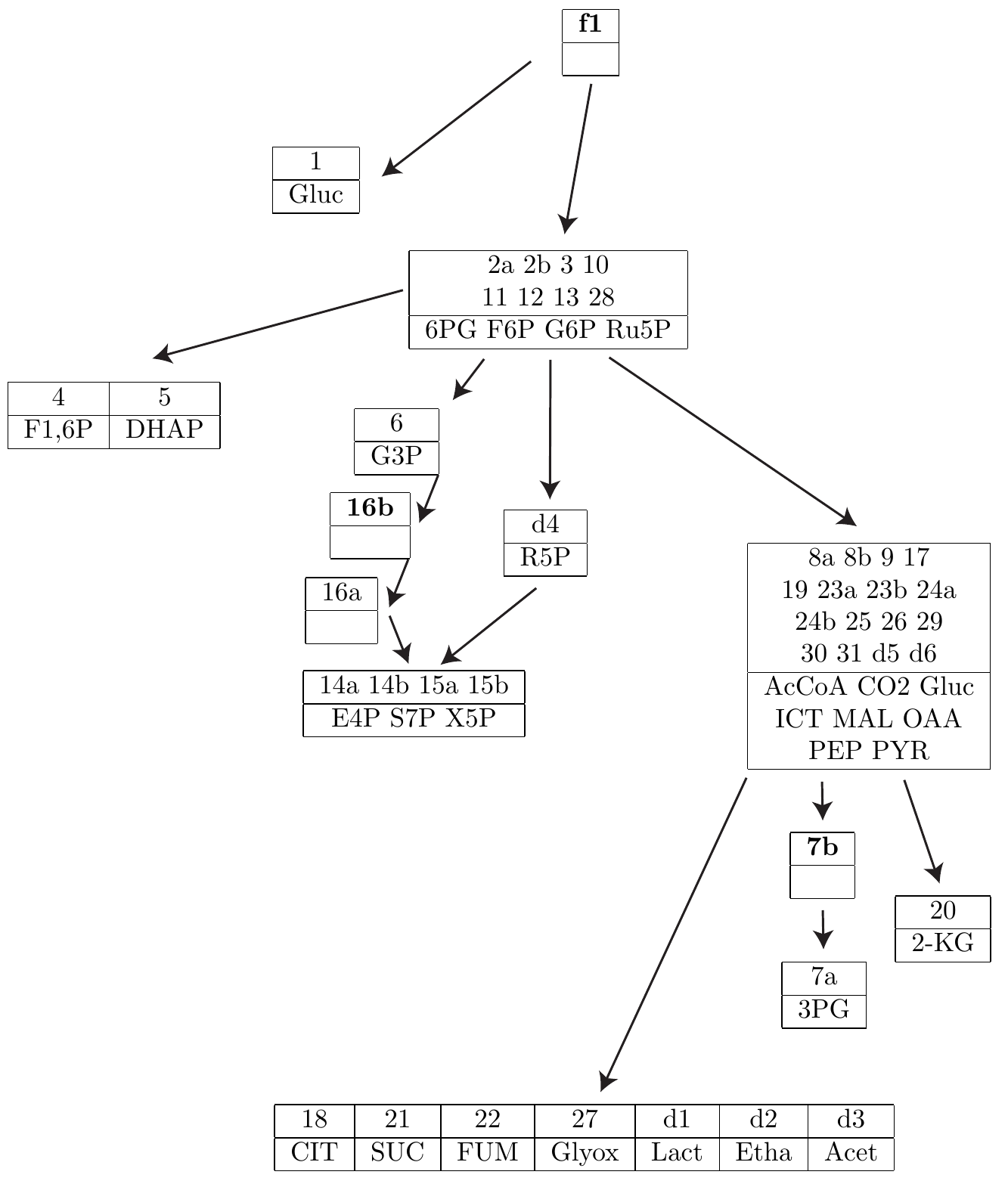}
\caption{\emph{The full flux influence graph of the TCA cycle, model variant $A$. Boldface reactions indicate self-influence $j^*\leadsto j^*$, for $j^*=f1, 7b, 16b$.
Arrows to box arrays with multiple columns point to each column separately.}}\label{fig:vanilla-tca}\label{fig:vanilla-tca-coalesced}
\end{figure}

As in fig.~\ref{fig:3.1}, we coalesce fanout-vertices into a single vertex table, albeit with several columns.
Fanout-vertices are vertices, which are subordinate to the same shared vertex in the full influence graph, and do not possess further influence.
Chains of single children provide natural examples which get coalesced.
See fig.~\ref{fig:vanilla-tca-coalesced}.

The monomolecular exit reactions \texttt{d1}-\texttt{d3} are single children, which prevent prevent perpetual increase of their mother metabolites \texttt{Lactate}, \texttt{Ethanol}, and \texttt{Acetate} by simple decay.
They are all activated by the \texttt{Glyoxylate} citric acid cycle in the lower part of fig.~\ref{fig:tca-net}.
Since single children have no influence, and since mere exit reactions do not enlarge the set $\mathcal M$ of metabolites, their addition or (formal) omission only adds or omits their own \texttt{child\{mother\}} box, subject to influence by others.
In fig.~\ref{fig:tca-abc} we see how the boxes of \texttt{d1}-\texttt{d3} are influenced by the same vertex box $\varphi \mathcal M^d(\varphi)$, and how their boxes have been coalesced to contribute three flow columns of one larger terminal vertex.

We discuss model variants $B,C$ of the same TCA cycle next, to explore the effect of additional monomolecular exit reactions.
The experimental study \cite{ishii}, which we call variant $B$, includes nine additional monomolecular exit reaction $\texttt{dd1}-\texttt{dd9}$.
For purely illustrative purposes, in variant $C$, we also add an artificial monomolecular exit reaction $\texttt{X1:}\ \texttt{Glucose}\to \cdot\ $.
 The resulting influence graphs are included in fig.~\ref{fig:tca-abc}. Note how the augmentation by the additional exit reaction $\texttt{X1}$ of model $C$ has a particularly strong coarsening effect on the influence structure.
We also provide the full influence graphs $\mathcal{F}$ for all three models in a single gray scale heat map, fig.~\ref{fig:tca-abc-table}.

\begin{figure}\centering
\includegraphics[width=\textwidth]{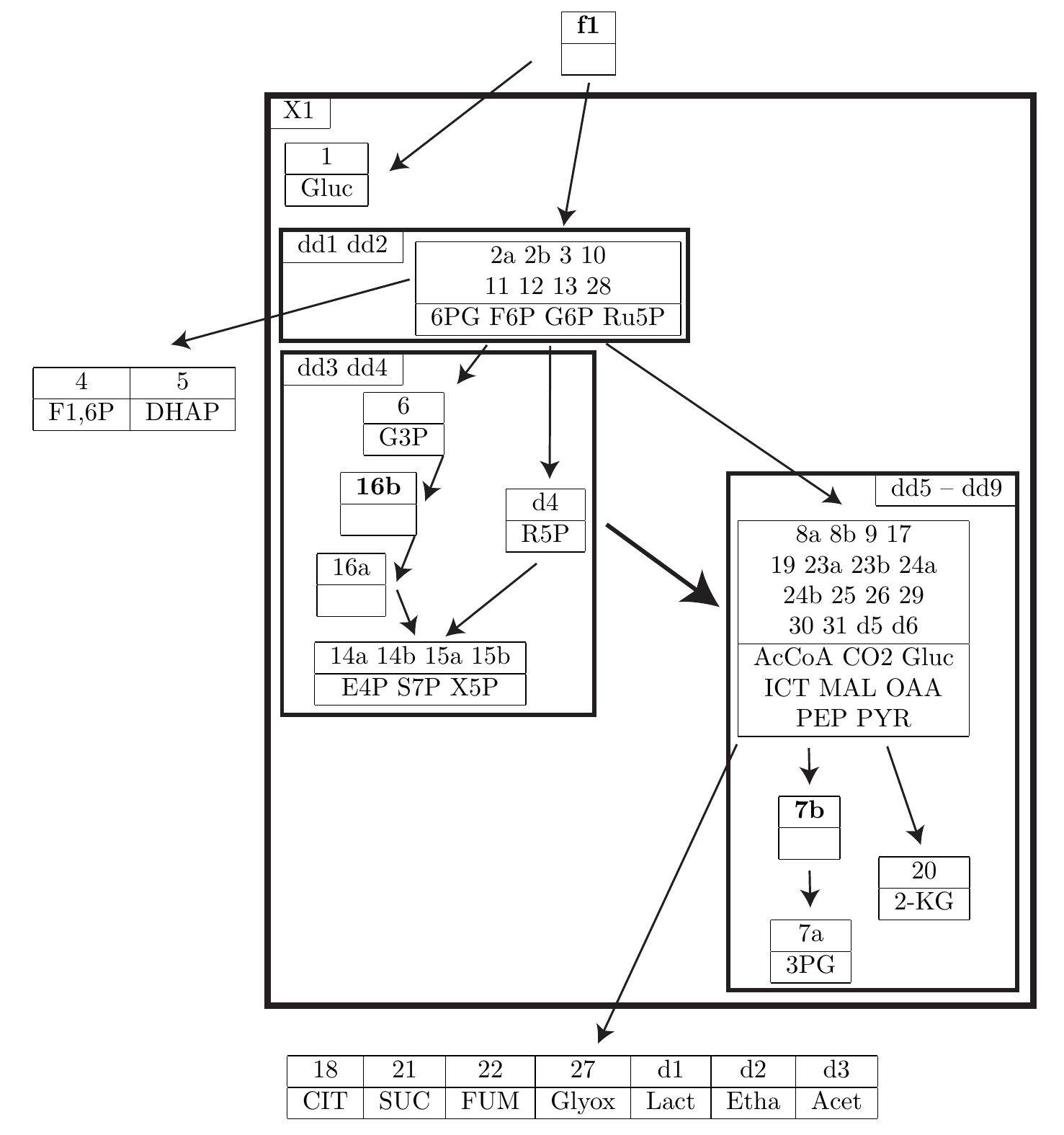}
\caption{\emph{The combined full flux influence graphs of the TCA cycles, model variants $A, B, C$. Notation as in fig.~\ref{fig:vanilla-tca}. The lumping effect of the monomolecular exit reactions \texttt{dd1}--\texttt{dd9}, added in model $B$, is indicated by bold frames. Such lumping frames merge all interior reactions and metabolites into a new vertex box $\varphi\mathcal M^d(\varphi)$. This includes the labels of the added reactions in their upper corners. Note the new bold arrow which is not inherited from model $A$. The artificial glucose exit \texttt{X1} of model $C$ generates the largest lump box.}}\label{fig:tca-abc}
\end{figure}

\begin{figure}
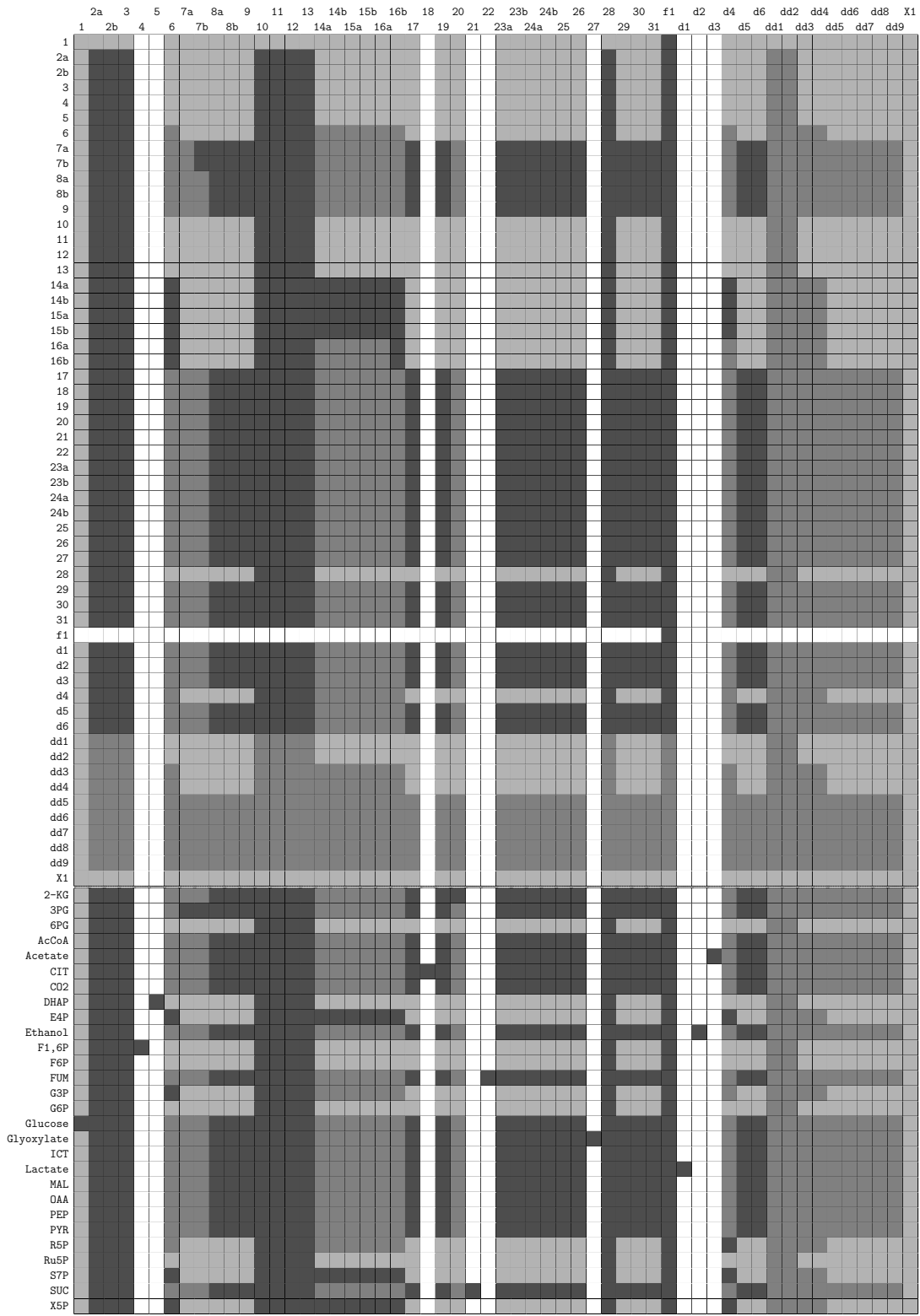

\resizebox{\textwidth}{!}{\large

\newcolumntype{P}[1]{>{\hspace{0pt}\rule{0pt}{\cellsize}}p{#1}}
\def\tmpCola{\cellcolor[gray]{1.0}} 
\def\tmpColb{\cellcolor[gray]{0.7}} 
\def\tmpCold{\cellcolor[gray]{0.5}} 
\def\tmpColh{\cellcolor[gray]{0.3}} 
\def\cellsize{2mm}
\def\bigcellsize{8mm}
\begin{tabular}{*{2}{P{\cellsize}}|*{56}{P{\cellsize}|}}
  &  \multicolumn{1}{l}{}  &  \multicolumn{1}{l}{}  &  \multicolumn{1}{c}{\makebox[0pt]{\texttt{2a}}}  &  \multicolumn{1}{l}{}  &  \multicolumn{1}{c}{\makebox[0pt]{\texttt{3}}}  &  \multicolumn{1}{l}{}  &  \multicolumn{1}{c}{\makebox[0pt]{\texttt{5}}}  &  \multicolumn{1}{l}{}  &  \multicolumn{1}{c}{\makebox[0pt]{\texttt{7a}}}  &  \multicolumn{1}{l}{}  &  \multicolumn{1}{c}{\makebox[0pt]{\texttt{8a}}}  &  \multicolumn{1}{l}{}  &  \multicolumn{1}{c}{\makebox[0pt]{\texttt{9}}}  &  \multicolumn{1}{l}{}  &  \multicolumn{1}{c}{\makebox[0pt]{\texttt{11}}}  &  \multicolumn{1}{l}{}  &  \multicolumn{1}{c}{\makebox[0pt]{\texttt{13}}}  &  \multicolumn{1}{l}{}  &  \multicolumn{1}{c}{\makebox[0pt]{\texttt{14b}}}  &  \multicolumn{1}{l}{}  &  \multicolumn{1}{c}{\makebox[0pt]{\texttt{15b}}}  &  \multicolumn{1}{l}{}  &  \multicolumn{1}{c}{\makebox[0pt]{\texttt{16b}}}  &  \multicolumn{1}{l}{}  &  \multicolumn{1}{c}{\makebox[0pt]{\texttt{18}}}  &  \multicolumn{1}{l}{}  &  \multicolumn{1}{c}{\makebox[0pt]{\texttt{20}}}  &  \multicolumn{1}{l}{}  &  \multicolumn{1}{c}{\makebox[0pt]{\texttt{22}}}  &  \multicolumn{1}{l}{}  &  \multicolumn{1}{c}{\makebox[0pt]{\texttt{23b}}}  &  \multicolumn{1}{l}{}  &  \multicolumn{1}{c}{\makebox[0pt]{\texttt{24b}}}  &  \multicolumn{1}{l}{}  &  \multicolumn{1}{c}{\makebox[0pt]{\texttt{26}}}  &  \multicolumn{1}{l}{}  &  \multicolumn{1}{c}{\makebox[0pt]{\texttt{28}}}  &  \multicolumn{1}{l}{}  &  \multicolumn{1}{c}{\makebox[0pt]{\texttt{30}}}  &  \multicolumn{1}{l}{}  &  \multicolumn{1}{c}{\makebox[0pt]{\texttt{f1}}}  &  \multicolumn{1}{l}{}  &  \multicolumn{1}{c}{\makebox[0pt]{\texttt{d2}}}  &  \multicolumn{1}{l}{}  &  \multicolumn{1}{c}{\makebox[0pt]{\texttt{d4}}}  &  \multicolumn{1}{l}{}  &  \multicolumn{1}{c}{\makebox[0pt]{\texttt{d6}}}  &  \multicolumn{1}{l}{}  &  \multicolumn{1}{c}{\makebox[0pt]{\texttt{dd2}}}  &  \multicolumn{1}{l}{}  &  \multicolumn{1}{c}{\makebox[0pt]{\texttt{dd4}}}  &  \multicolumn{1}{l}{}  &  \multicolumn{1}{c}{\makebox[0pt]{\texttt{dd6}}}  &  \multicolumn{1}{l}{}  &  \multicolumn{1}{c}{\makebox[0pt]{\texttt{dd8}}}  &  \multicolumn{1}{l}{}  &  \multicolumn{1}{c}{\makebox[0pt]{\texttt{X1}}}  \\
\end{tabular}
}
\caption{\emph{Flux influence relations in the three variants $A,B,C$ of the TCA cycle.
As in fig.~\ref{tbl:3.1}, influence $\alpha \leadsto \beta$ is indicated in column $\alpha \in \mathcal{E}$ and row $\beta \in \mathcal{E} \cup \mathcal{M}$.
Gray scales encode the influence in the respective model variant.
Here \cbox{black!0}\,:
no influence in any variant;
\cbox{black!30}\,:
influence only in model $C$;
\cbox{black!50}\,:
influence in models $B, C$;
\cbox{black!70}\,:
influence in all three models $A,B,C$.}}
\label{fig:tca-abc-table}
\end{figure}

The addition of further monomolecular exit reactions $\texttt{dd1},\ldots,\texttt{dd9}$ and \texttt{X1} in model variants $B$ and $C$, respectively,
lumps (alias, merges or collapses) certain vertices of model $A$ and successively coarsens the full flux influence graphs, see fig.~\ref{fig:tca-abc}. Indeed, reactions $\langle\texttt{dd1},\texttt{dd2}\rangle$ in $B$ have enlarged the previous upper vertex $\langle \texttt{2a},\texttt{2b},\ldots,\texttt{28}\rangle$ of model $A$, fig.~\ref{fig:vanilla-tca}.
The vertices of reactions $ \texttt{6}, \texttt{16a}, \langle\texttt{16b}\rangle$, $\texttt{d4}$, and $ \langle\texttt{14a}, \texttt{14b}, \texttt{15a}, \texttt{15b}\rangle$ in $A$ have been lumped into a single vertex, in model $B$, augmented by $ \texttt{dd3},\texttt{dd4}$.
The remaining exit reactions $\texttt{dd5},\ldots,\texttt{dd9}$ of $B$ lump, and augment, the vertices $ \langle\texttt{8a},\texttt{8b},\ldots,\texttt{d6}\rangle$ and $ \texttt{7a}, \langle\texttt{7b}\rangle, \texttt{20}$ from model $A$. 
Reactions \texttt{7a} and \texttt{20}, for example, lose their single child status due to monomolecular exit reactions \texttt{dd5} and \texttt{dd9}, respectively. 
We conclude that the lumping caused by the monomolecular exit reactions of model B, in the TCAC reaction network of fig.~\ref{fig:tca-net}, emphasizes the grouping into phosphorylation, by the upper two vertices, versus the large lower vertex of the citric acid cycle.

\texttt{Glucose} is the central driving feed metabolite of the entire network. 
The artificial addition, in model $C$, of a new monomolecular exit reaction \texttt{X1} of $\texttt{Glucose}$ forces even stronger lumping. 
All vertices of model $B$, except for some single children, are lumped into a single large vertex of reactions $\langle \texttt{1},\ldots,\texttt{X1}\rangle$. 
In fact, \texttt{Glucose}, the single product of the single feed reaction \texttt{f1}, has lost its single child mother-status of reaction $\texttt{1}$ in the initializing chain of reactions \texttt{f1} and \texttt{1}. Perturbations of the two \texttt{Glucose} children \texttt{1} and \texttt{X1} therefore influence each other, and the \texttt{Glucose} level itself. 
In models $A$ and $B$ this could only be achieved, equivalently, by a perturbation of the driving feed rate parameter \texttt{f1} -- with sweeping influence on the whole network. 
It is therefore essential -- and has been painstakingly observed in the careful and pioneering knockout experiments by \cite{ishii, ishii2} -- to meticulously control the level of \texttt{Glucose} uptake in order to obtain any meaningful results.

Fig.~\ref{fig:tca-abc} summarizes the results of our model comparison $A$ -- $C$, in view of augmenticity. 
Innermost boxes show the full influence graph of model $A$. 
Since no new metabolites have been added in models $B,C$, augmenticity theorem \ref{thm:B4.1} with $\mathcal M_1=\mathcal M_0$ implies two possibilities. 
First, new influences involving the added reactions may lump existing vertices into larger vertex boxes. 
This is indicated by larger boldface frames around the finer structures of model $A$. 
Second, new hierarchic arrows may appear between the larger frames. 
The new hierarchy in the augmented model must be compatible with the ordering in the smaller model. 
Nevertheless, additional influence arrows may, and do, appear in the augmented model, which are not already implied by the ordering in the smaller model.
These additional influence arrows between frames are drawn to originate or terminate at the lumping frames. 
They are drawn to not reach the boxes contained inside each frame. 
This distinguishes the augmented influence arrows from the pre-existing ones.

\begin{figure}\centering
\includegraphics[width=\textwidth]{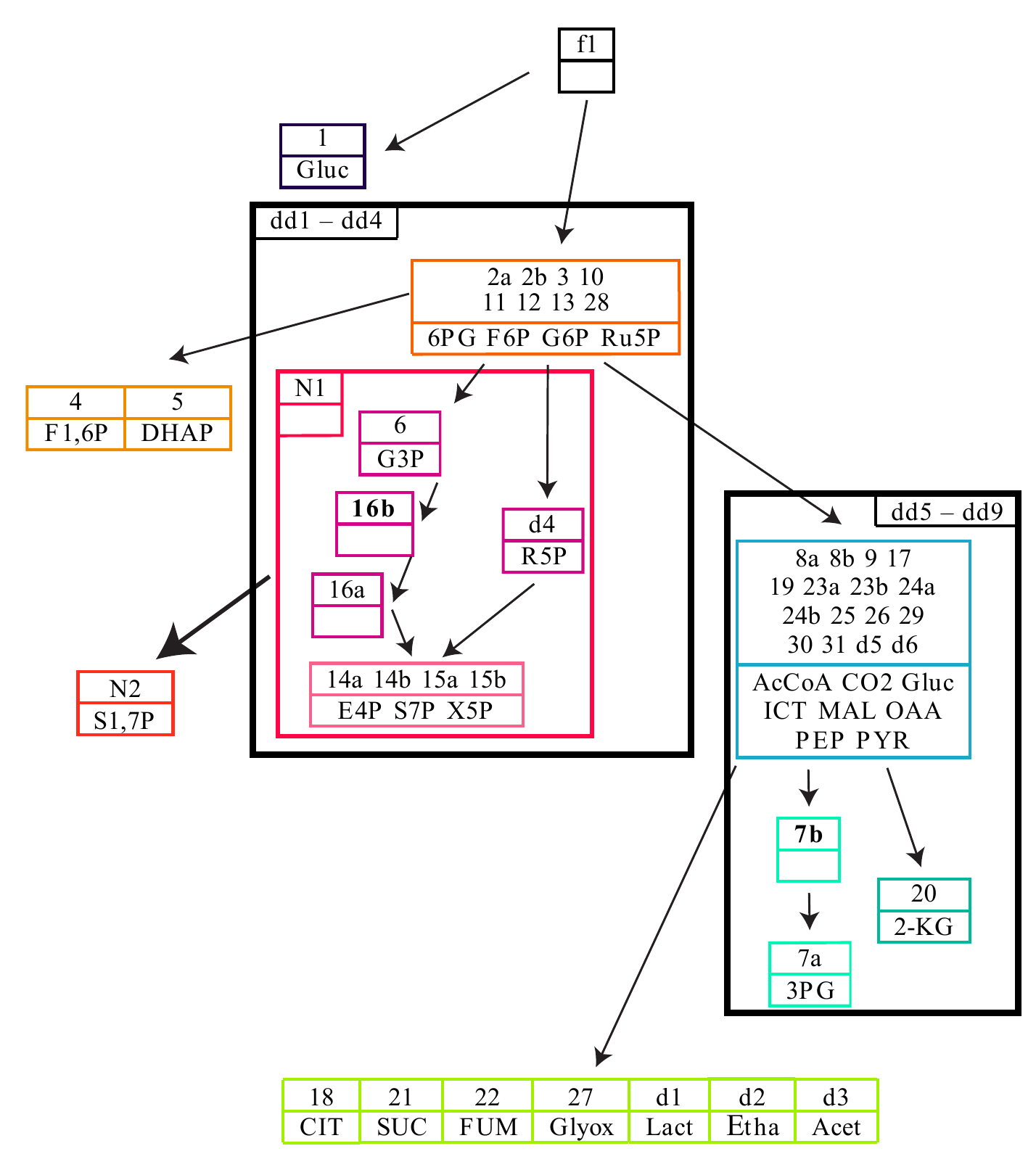}
\caption{\emph{The full flux influence graph of the TCA cycle, model invariants $A, D, E$.
Notation as in fig.~\ref{fig:vanilla-tca}.
The lumping effect of the new reactions $N1, N2$ and the new metabolite $S1,7P$, in model $D$ is indicated by the bold lumping box with label $N1\{~\}$.
Note the new bold arrow from that box to the new vertex $N2\{ S1, 7P\}$.
The lumping effect of the monomolecular exit reactions \texttt{dd1}--\texttt{dd9}, in model $E$, produces two additional lumping boxes with labels \texttt{dd1}--\texttt{dd4} and \texttt{dd5}--\texttt{dd9}, respectively, but without any additional arrows.
Color coding (online) emphasizes the successive lumping of these boxes.
The box structure is also represented in the metabolic network of fig.~\ref{fig:tca-net}, by matching colors.}}\label{fig:tca-ade}
\end{figure}

Alternatively, we can visualize the augmenticity properties of the influence structures $A$ -- $C$ in the heat map of fig.~\ref{fig:tca-abc-table}.
All three models share the same metabolite set $\mathcal{M}$.
The successive augmentations $\mathcal{E}_A \subset \mathcal{E}_B \subset \mathcal{E}_C$ of the respective reaction sets $\mathcal{E}$, therefore, only add influences, successively, but never remove any.
See augmenticity theorem~\ref{thm:B4.1}.
Consider influence $\leadsto$, or non-influence $\not\leadsto$, of column $\alpha \in \mathcal{E}$ on row $\beta \in \mathcal{E} \cup \mathcal{M}$.
Then only the following four triples, with their respective gray scales at position $\beta \alpha$, can appear for models $(A, B, C)$:
the white cell $\cbox{black!0}=(\not\leadsto,\not\leadsto,\not\leadsto)$,
the light gray cell $\cbox{black!30}=(\not\leadsto,\not\leadsto,\leadsto)$,
the medium gray cell $\cbox{black!50}=(\not\leadsto,\leadsto,\leadsto)$,
and the dark cell $\cbox{black!70}=(\leadsto,\leadsto,\leadsto)$.
The reduced Ishii model $A$ provides the most sparse, dark influence structure.
The full Ishii model $B$, with all monomolecular exit reactions included, provides the less sparse influence structure which adds the medium gray cells.
The artificial monomolecular exit \texttt{X1} of \texttt{Glucose} in model variant $C$, finally, adds all light gray cells.
The influence matrix of model $C$ becomes so crowded by entries of nonzero influence $\leadsto$ that the influence structure alone becomes visibly meaningless for the purpose of any functional understanding of the metabolite network. 
Indeed it is the original Ishii model $B$, which represents functionality in the  TCA cycle best -- better, in fact, than the less complete model $A$. 
The preference for model $A$, in \cite{fiedler15example}, was due to a lack of computational efficiency in the adhoc symbolic treatment which required the somewhat arbitrary omission of exit reactions \texttt{dd1},\ldots,\texttt{dd9} -- in spite of an investment of substantial raw computing power.

In fig.~\ref{fig:tca-ade} we study the augmentation of the reduced Ishii model $A$ by the new metabolite \texttt{S1,7P} and reactions \texttt{N1}, \texttt{N2} due to Nakahigashi~et~al; see \cite{ishii2} and model $D$. Model $E$ further augments $D$ by the monomolecular exit reactions $\texttt{dd1},\ldots,\texttt{dd9}$ of the full Ishii model $B$. The models $D,E$ augment model $A$ by the single metabolite 
\begin{equation}
\mathcal M_1\smallsetminus \mathcal M_0=\{\texttt{S1,7P}\},
\end{equation}
and at least reactions $\{\texttt{N1},\texttt{N2}\} = \mathcal E_D\smallsetminus \mathcal E_A$.
The choice $J^\vee(\texttt{S1,7P})=\texttt{N2}$, in theorem \ref{thm:B4.1}, makes the $1\times 1$ determinant \eqref{eq:B4.3} nonsingular. Therefore our previous comments on the augmentation sequence of models $A,B,C$ apply to $A, D, E$, verbatim.

More specifically, the augmentation of model $A$ by model $D$ lumps the phosphorylation branch into a single influence vertex box, augmented by the new reaction \texttt{N1}. The lumping effect is similar to the addition of exit reactions \texttt{dd3, dd4} in the Ishii model.
Reaction \texttt{N1}, in turn, produces the new metabolite \texttt{S1,7P}. The flux of the onwards reaction \texttt{N2} with educt \texttt{S1,7P} is influenced by the lumped box with label \texttt{N1}.

Further monomolecular exit reactions $\texttt{dd1},\ldots, \texttt{dd9}$ augment model $D$ to model $E$, without new metabolites.
Exit reactions $\texttt{dd1},\ldots, \texttt{dd4}$ of model $E$ simply get lumped into the phosphorylation vertex of model $D$.
The remaining exit reactions $ \texttt{dd5}, \ldots, \texttt{dd9}$, as before,  lump the citric acid cycle of $\langle \texttt{8a}, \texttt{8b},\ldots,\texttt{d6}\rangle$ with its unidirectional influences on $\texttt{7a},\langle\texttt{7b} \rangle, \texttt{20}$ into a new vertex of full mutual influence.


\section{Conclusions and discussion}
\label{sec:disc}

In conclusion, the flux influence graphs presented here, and their augmenticity, 
are designed to provide first steps towards a mathematically sound analysis of the 
sensitivity dependencies in general, multimolecular metabolic networks. 
Our results rely on the network structure, only. 
They hold true, universally and in a precise algebraic sense, for almost all choices of 
rate functions and their parameters.
Our approach is reliably automated, and still is able to assist in a fast and meaningful  
first conceptual analysis of metabolic networks -- even in the hands of biological 
non-experts like ourselves.

We have presented five types of results. 
\begin{itemize}
 \item Theorems \ref{thm:2.1} -- \ref{thm:2.3} are based on the notion of \emph{child selections}, i.e.~injective maps $J: \mathcal{M} \rightarrow \mathcal{E}$ such that each ``mother'' metabolite $m$ is an input $m \vdash J(m)$ of the ``child'' reaction $J(m)$. 
 The existence of certain child selections implies linear nondegeneracy of the network, flux influences $j^* \leadsto j'$, and metabolite influences $j^* \leadsto m'$, respectively, under perturbations of the rate function $r_{j^*}(x)$ of reaction $j^*$. 
 Each theorem in fact asserts certain quantities to be nonzero, algebraically. 
 In other words, the relevant quantities turn out to be nonzero, when viewed as polynomial or rational expressions in the partial derivatives $r_{jm} = \partial_{x_m}  r_j(x)$ of the rate functions $r_j(x)$ with respect to the metabolite concentrations $x_m$ of the reaction inputs $m \vdash j$. 
 
 \item Theorems \ref{thm:2.4} -- \ref{thm:2.7} establish a \emph{hierarchy} of flux and metabolite responses to rate perturbations of any reaction $j^*$. 
 Only the crucial property of \emph{transitivity} of flux influence allows us to comprehensively summarize all flux and metabolite responses in a single \emph{flux influence graph} $\mathcal{F}$.
 
 \item Theorem \ref{thm:B4.1} investigates \emph{augmenticity}. 
 This allows us to systematically predict the lumping effects of network extensions $(\mathcal{E}_1, \mathcal{M}_1) \supseteq  (\mathcal{E}_0, \mathcal{M}_0)$ on the associated flux influence graphs.
 
 \item Section \ref{sect:algo} provides an \emph{efficient algorithm} to implement our results for the moderately large networks in data bases like \cite{KEGG, BernhardD}. 
 Our algorithm is based on the Schwartz-Zippel lemma \ref{lem:schwartz-zippel}. 
 For each network it involves a single randomized computation for integer arithmetic modulo moderately large primes. We have tested our approach for networks which involve up to a few hundred metabolites and reactions. 
 
 \item Specific \emph{examples} have been presented. In section \ref{sec3:theoex} we have addressed two theoretical examples to illustrate our results. Section \ref{sec7:extca} has explored the TCA cycles proposed by \cite{ishii, ishii2}, together with several variants which only differed in monomolecular exit reactions of certain metabolites. It was evident how the inclusion, or omission, of monomolecular exit reactions may strongly affect the resulting hierarchy of flux influences in the network.
 \end{itemize}

We did not aim for quantitative results. Given a rate perturbation of any reaction $j^*$, our results only distinguish zero responses of steady states from nonzero responses. 
Unlike \cite{giordano2015computing}, we did not even keep track of the (positive or negative) sign of any nonzero response. 
While a zero response is exact, and independent of the particular reaction rates involved, we repeat that nonzero responses only hold in an algebraic sense.

Our \emph{assumptions} are few. Surjectivity, i.e. full rank $M=|\mathcal{M}|$, of the stoichiometric matrix $S: \mathcal{E} \rightarrow \mathcal{M}$, is our main requirement. Structural assumptions on the network are specified in terms of child selections, as repeated above. For the rate functions we only assume algebraic independence of the partial derivatives $r_{jm} = \partial_{x_m}  r_j(x)$. We do not require any further assumptions. In particular, no numerical information is required.

Our results can be read as pure linear algebra with an abstract underlying ``network structure''. From an abstract point of view, no reference to any steady state $x$ is required. To interpret our results in the context of a steady state response to external rate perturbations in concrete reaction networks, of course, at least the existence of a steady state $x$ as in \eqref{eq:1.10} is required. Since actual reaction rates may become linearly degenerate in absence of reactants, it will also be safer to consider strictly positive steady state components, $x_m >0$, in such cases, to actually provide algebraically independent partial derivatives $r_{jm}$.

We have already pointed out how \emph{mass action kinetics} \eqref{eq:1.20} are not sufficiently rich, in general, to allow for algebraically independent rate coefficients $r_{jm}$.
A large complementary body of work on multiplicity of steady states and on network sensitivity, specifically in the context of mass action kinetics, has been initiated by Feinberg and coworkers; see \cite{feinberg95, sf09, sf10, sf11, shietal11}.
Their emphasis is on ``quantitative robustness'' of the steady state metabolite response to perturbations of fluxes and some ``elemental'' metabolite concentrations. Robustness is understood as bounded, usually small, but nonzero linear response to the external perturbation.

In the present paper we did not treat stoichiometric subspaces. 
Suppose our surjectivity assumption for the stoichiometric matrix $S: \mathcal{E} \rightarrow \mathcal{M}$ is violated. Let $Q: \mathcal{M} \rightarrow \mathcal{M}$ project onto the nontrivial co-kernel of $S$, i.e.~$QS=0$. 
Then \eqref{eq:1.8} implies $Q \dot{x}(t) = 0$, and hence $Qx(t) = \eta = \mathrm{const}$. 
The invariant, affine linear subspaces defined by the constant stoichiometric integrals $\eta$ are called \emph{stoichiometric subspaces}. It is not difficult to adapt our approach to the presence of nontrivial stoichiometric subspaces. 
Indeed, we may easily split our matrix $B$, in section \ref{sec3:theoex}, to account for the range $(\mathrm{id}_{\mathcal{M}}-Q)S$ and the stoichiometric integrals $\eta$, separately. 
Moreover, a thorough analysis should include the influence of the stoichiometric integrals $\eta$, as parameters, together with any resulting transitivity properties. 
To avoid the foreseeable inflation of notation, in the present paper, and for a certain lack of concrete biological motivation, we chose not to include that case here. 
Omnipresent monomolecular exit reactions $X_m \rightarrow \cdot\ $ are able to destroy any stoichiometric subspaces, anyway.

\emph{Monomolecular exit reactions}, however, raise a cautioning menetekel. 
We illustrate this serious caveat by the following construction. 
Consider any network $(\mathcal{E}_0, \mathcal{M})$. 
Extend to a network $(\mathcal{E}^\times, \mathcal{M}) \supseteq  (\mathcal{E}_0, \mathcal{M})$ such that $\mathcal{E}^\times$ provides a monomolecular exit reaction $j_m \in \mathcal{E}^\times$, for each metabolite $m \in \mathcal{M}$, in addition to the reactions and (non-identical) exit reactions already present in $\mathcal{E}_0$. 
In the original network we may assume, of course, that each metabolite $m$ participates in at least one none-exit reaction. Else we may omit that metabolite $m$ from consideration. 
In particular, then, the extended network $(\mathcal{E}^\times, \mathcal{M})$ does not contain any single children. 
Also note how an appropriate choice $r_{j_m}(x) := 1- k_{j_m}x_m$ of exit reaction rates will not even affect a given steady state of $(\mathcal{E}_0, \mathcal{M})$. 
Our final assumption on the original network is that each reaction $j' \in \mathcal{E}_0$ possesses a mother metabolite $m'$ which does not just participate catalytically in reaction $j'$. 
Else $\bar{y}^{j'} = y^{j'}$, and we may omit that reaction $j'$ from consideration, as well.

The exit choice $J^\times(m) := j_m$ of a child selection shows regularity of the extended network; see theorem \ref{thm:2.1}. 
The same choice establishes self-influence $j^* \leadsto j^*$, in  the extended network, for any non-exit reaction $j^*$ of the original network $\mathcal{E}_0$; see theorem \ref{thm:2.2}(i).

More seriously, consider any two distinct non-exit reactions $j^* \neq j'$ in the original network $\mathcal{E}_0$. 
We claim \emph{universal mutual flux influence} $j^* \leadsto j'$, in the extended network $(\mathcal{E}^\times, \mathcal{M})$. 
To show our claim, we first pick a non-catalytic mother metabolite $m'$ of reaction $j'$. Thus $\bar{y}^{j'}_{m'} - y^{j'}_{m'} \neq 0$. 
Define  a child selection $J$ in the extended network, as follows. 
Let $J(m') := j'$, and pick monomolecular exit reactions $J(m) := J^\times (m) = j_m$ for all other metabolites $m \neq m'$. 
Then theorem \ref{thm:2.2}(ii) proves our claim. 
Indeed, $j' =J(m') \in J(\mathcal{M})$, by construction. 
Further, $j^* \not\in J(\mathcal{M})$ because $j^* \neq j'$ and $j^*$ is not an exit reaction. 
Finally, consider the columns of the stoichiometric matrix $S$ restricted to the swapped reaction set $\{ j^*\} \cup J(\mathcal{M}) \smallsetminus \{j'\}$. 
By construction, and because the only non-exit column $S^{j'}$ of the restriction satisfies $S^{j'}_{m'} = \bar{y}^{j'}_{m'} - y^{j'}_{m'} \neq 0$, that restriction possesses nonzero determinant. 
This establishes assumption \eqref{eq:2.6} of theorem \ref{thm:2.2}(ii), and hence establishes our claim $j^* \leadsto j'$ of universal mutual flux influence, in the extended network $(\mathcal{E}^\times, \mathcal{M})$.

Similarly, and in the same extended network $(\mathcal{E}^\times, \mathcal{M})$, we can show \emph{universal metabolite influence} $\mathcal{E}_0 \ni j^* \leadsto m'$, for any metabolite $m' \in \mathcal{M}$ which is not purely catalytic in the original network $(\mathcal{E}_0, \mathcal{M})$.

The above elementary, but far-reaching, menetekel shows how the presence of mono\-molecular exit reactions, for all metabolites, catastrophically lumps all non-exit reactions into a single flux equivalence class. 
In particular, any interesting hierarchy is obliterated and all sparsity is effaced. 
The present qualitative theory therefore requires meticulous attention as to the actual scales -- but not the actual numerical values -- of decay rates of metabolites, relative to the reaction rates within the metabolic network.

Such simple examples underline why the purely \emph{monomolecular case}, where each stoichiometric vector $y^j, \bar{y}^j$ has singleton or empty support, rightly commands intense interest.  
In the purely monomolecular case, the reaction network itself defines a directed \emph{reaction graph} with single metabolites, as vertex set $\mathcal{M} \cup \{0\}$, and reaction arrows, as edges $\mathcal{E}$. 
The reaction graph essentially coincides with Feinberg's graph of reaction complexes, in that case, and exhibits Feinberg deficiency zero; see \cite{feinberg95} for terminology, and \cite{fiedler15mono} for a proof of this claim. 
In fact our previous monomolecular results in \cite{fiedler15mono} were expressed, and proved, via cycles and spanning trees in the reaction graph. 
For much more elegant formulations of these results, and significant further developments, see \cite{Vassena15, Vassena16, MatanoVassena16}.
 
In the monomolecular reaction graph setting, for example, the kernel condition \eqref{eq:2.4b} of theorem~\ref{thm:2.1} guarantees that the child selection $J$ does not run into oriented or nonoriented cycles. 
 Equivalently, $J$ defines a directed exit path $\gamma^0$ from the unique mother vertex $m^*$ of any reaction $j^*$ to the exit vertex $0$. 
 Theorem~\ref{thm:2.2} then characterizes $j^* \leadsto j'$ by the existence of a second directed path $\gamma'$ from the same mother $m^*$ of $j^*$ to $j'$. 
 The paths $\gamma^*$ and $\gamma'$ are disjoint, except for $m^*$. 
 One of the two paths contains $j^*$.
The results of \cite{fiedler15mono} can therefore be derived from the multimolecular versions in theorems~\ref{thm:2.1} and~\ref{thm:2.2}, directly.
More importantly, regions of influence can be described in purely graph theoretic language. 
For deeper considerations in this direction, see \cite{Vassena15, Vassena16, MatanoVassena16} again.
This eliminates the necessity of any symbolic algebra at all.

We therefore find the interpretation of monomolecular metabolic networks as reaction graphs extremely appealing and intuitive. 
We have already mentioned several multimolecular graph paradigms due to Feinberg and his coworkers. 
Another option, in the present multimolecular case, are bipartite digraphs with respective vertices $\mathcal{M} \cup \{ 0\}$ and $\mathcal{E}$. 
Directed edges $\mathcal{M} \cup \{ 0\} \ni m \rightarrow j \in \mathcal{E}$ then indicate inputs $m \vdash j$, or feed reactions $j$. 
The opposite edges $j \rightarrow m$ refer to outputs, or exit reactions.
So far, however, we neither succeeded via bipartite reaction graphs, nor via extensive child swapping, to provide alternative proofs of transitivity of flux or metabolite influence in the multimolecular case.

Our results above have been of a local perturbation character. 
Indeed we have invoked the standard implicit function theorem, in section \ref{sec1:intro}, to justify our linear algebra setting. 
See \eqref{eq:1.13}. 
Knockout experiments, on the other hand, typically obstruct some reaction $j^*$ in the network entirely, and study the consequences for steady states. 
Results for \emph{large perturbations} are therefore required.

Fortunately, our results on zero flux or metabolite responses are not limited to small perturbations, because they only rely on structural information concerning the network stoichiometry. 
Therefore they apply to large perturbations and to knockout experiments, as well. 
To be specific, consider any two steady states
	\begin{equation}
	0=S r(\varepsilon^\iota, x^\iota), \qquad r_{j^*}(\varepsilon^0,x^0) \neq r_{j^*}(\varepsilon^1,x^1)\,,
	\label{eq:disc.22}
	\end{equation}
 with $\iota =0,1$, which are associated to a large parameter perturbation $\varepsilon^1 - \varepsilon^0 = e_{j^*}\,$. 
Let $\delta x^{j^*}$:= $x^1-x^0\neq 0$ denote a resulting large steady state perturbation.
We interpolate linearly between parameters $\varepsilon^0$ and $\varepsilon^1$, and between steady states $x^1$ and $x^0$, by
	\begin{equation}
	\varepsilon^\vartheta := \varepsilon^0 + \vartheta \cdot e_{j^*} \quad \mathrm{and} \quad
	x^\vartheta := x^0 + \vartheta \cdot \delta x^{j^*}\,,
	\label{eq:disc.23}
	\end{equation}
for $0 \leq \vartheta \leq 1$. 
The intermediate evaluation points $(\varepsilon^\vartheta, x^\vartheta)$ are not required to be steady states.
We obtain
	\begin{equation}
	0= S(\eta e_{j^*} + \tilde{R}\delta x^{j^*})\,.
	\label{eq:disc.24}
	\end{equation}
Here the scalar $\eta$ and the matrix $\tilde{R}= (\tilde{r}_{jm})$ abbreviate integrals
	\begin{equation}
	\eta = \int_0^1 \partial_{\varepsilon_{j^*}}r_{j^*}
	(\varepsilon_{j^*}^\vartheta, x^\vartheta) d \vartheta\,,\qquad
	\tilde{r}_{jm} = \int_0^1 r_{jm}
	(\varepsilon_j^\vartheta, x^\vartheta) d \vartheta\,.
	\label{eq:disc.25}
	\end{equation}
Comparing \eqref{eq:disc.24}, for $\eta=1$, with our original setting \eqref{eq:1.13}, there is no difference. 
Therefore we obtain identical results on the influence of large perturbations, for generic rate functions $r_j(x)$, as for the small local perturbations considered so far.

More precisely, suppose first that $j^* \in \mathcal{E}$ does not influence $\alpha \in \mathcal{E} \cup \mathcal{M}$, algebraically. Then this zero influence holds, in \eqref{eq:disc.24} as in \eqref{eq:1.13}, independently of the values of $\tilde{r}_{jm}\,$. Hence zero influence persists under large perturbations.

Next, assume local metabolite influence $j^* \leadsto m'$. 
We then claim that the large perturbation \eqref{eq:disc.22} satisfies $\delta x^{j^*}_{m'} \neq 0\,$, as well, for generic reaction rate functions. Suppose therefore $\delta x^{j^*}_{m'} = 0$.
Let us slightly perturb the derivatives $\partial_m r_j(\varepsilon^0,x^0)$ of the rate functions at $x^0$, but not any of the rates $r_j(\varepsilon^0,x^0),\,r_j(\varepsilon^1,x^1)$ themselves. 
This keeps $\delta x^{j^*}_{m'} = 0$ unchanged, but avoids any algebraic degeneracies of the derivatives $r_{jm} = \partial_m r_j(\varepsilon^0,x^0),$ at the steady state $x^0$. 
Since we have assumed $r_{j^*}(\varepsilon^0,x^0) \neq r_{j^*}(\varepsilon^1,x^1)$ in \eqref{eq:disc.22}, we may now perturb the reaction rate $r_{j^*}(\varepsilon^0,x^0)$ at $x^0$, very slightly, but keep the same rate fixed at $x^1$. 
Our very slight perturbation of the reaction rate $r_{j^*}(\varepsilon^0,x^0)$ will therefore nudge $x^0_{m'}$ and $\delta x^{j^*}_{m'}\,$, but not $x^1_{m'}$.
By definition, local influence $j^* \leadsto m'$ implies $z^{j^*}_{m'} \neq 0\,$, algebraically, for the local response vector $Bz^{j^*} = -e_{j^*}$ at $x^0$. Therefore our very slight second perturbation will nudge $\delta x^{j^*}_{m'}$ away from zero, by a very slight multiple of $-z^{j^*}_{m'} \neq 0$.
This establishes persistence of metabolite influence under larger perturbations, for generic rate functions. 
We omit analogous arguments, which establish generic persistence of flux influence under large perturbations.


Large perturbations, in the form of knockout experiments, are currently used to reveal and test the structure of metabolic networks. 
More ambitiously, however, large perturbations also offer an option for the active \emph{control of metabolic networks}. See the monograph \cite{HeinrichSchuster} for an early background.
The full flux influence graph $\mathcal{F}$ can serve as a guiding principle. 
First of all, the influence graph reliably identifies some modularity, alias regions of influence, in terms of biological function. 
Our discussion of the TCA cycle and its variants, in section \ref{sec7:extca}, illustrates and emphasizes that aspect. 
It is therefore conceivable to couple such functional units, from different networks, and study their hierarchic or mutual interaction. 
Augmenticity, as in section \ref{secB4:aug}, provides a key aspect of this second step.

More modestly, and more specifically, we may consider a second network which just provides a key catalytic enzyme, as exiting output, to significantly perturb the rate of some reaction $j^*$ in a primary network. 
Already the full flux influence graph $\mathcal{F}$ of the primary network will point at desirable target metabolites $m'$ in the primary network which may thus be stimulated, or inhibited. 
At the same time, $\mathcal{F}$ will reliably delimit the collateral influence of such a unidirectional coupling. 
As a drawback of our current results, however, we have not determined the precise signs of the resulting influences. 
We have only mentioned the determination of response signs by \cite{giordano2015computing}, along with the computational cost involved, in section \ref{sect:algo}.

An elegant \emph{upper estimate} of influence regions in multimolecular networks $(\mathcal{E}, \mathcal{M})$ has been obtained earlier by Okada; see \cite{OkadaFiedlerMochizuki15} and \cite{OkadaMochizuki16}. 
Although the Okada result includes the treatment of stoichiometric subspaces, let us explain its relation to our present paper in the case of surjective stoichiometric matrices $S$. 
Okada considers subnetworks $(\mathcal{E}_0, \mathcal{M}_0) \subseteq (\mathcal{E}, \mathcal{M})$ which are ``output complete'', i.e.
	\begin{equation}
	\mathcal{M}_0 \ni m_0 \vdash j_0 \in \mathcal{E}\quad \mathrm{implies} \quad j_0 \in \mathcal{E}_0\,. 
	\label{eq:ok1}
	\end{equation}
In slightly cryptic form, Okada also requires
	\begin{equation}
	\mathrm{dim} \,(\mathrm{ker}\, S \cap \mathcal{E}_0) + \mathrm{dim}\, \mathcal{M}_0 = \mathrm{dim}\, \mathcal{E}_0 \, . 
	\label{eq:ok2}
	\end{equation}
As a consequence, Okada obtains the following upper estimates on the flux influence sets $I_\mathcal{E}(j^*)$ and the metabolite influence sets $I_\mathcal{M}(j^*)$ introduced in \eqref{eq:2.10} and \eqref{eq:2.11}:
	\begin{equation}
	\underset{j^* \in \mathcal{E}_0}{\bigcup} I_\mathcal{E}(j^*) \subseteq \mathcal{E}_0 \quad \mathrm{and} \quad 
	\underset{j^* \in \mathcal{E}_0}{\bigcup} I_\mathcal{M}(j^*) \subseteq \mathcal{M}_0 \,.
	\label{eq:ok3}
	\end{equation}

Note how Okada is careful not to claim any actual influence of any reaction $j^*$ in particular.
For example, $\mathcal{E}_0$ may contain single children $j^*$. 
We already saw in sections \ref{sec2:mainres} how single children $j^*$ exert no influence, at all, except on their own single-child mother $m^* \vdash j^*$. 
But empty influence sets are perfectly compatible with Okada's upper estimate \eqref{eq:ok3}, indeed.

The beauty of Okada's upper estimate \eqref{eq:ok3} on regions of influence is not diminished by the simplicity of its proof. 
In terms of the $M \times E$ block matrix $B$ which we have introduced in section \eqref{eq:3.2}, Okada's output completeness condition \eqref{eq:ok1} implies that $B$ restricts to a linear map
	\begin{equation}
	B: \quad  (\mathrm{ker}\,S \cap \mathcal{E}_0) \cup \mathcal{M}_0 \ \longrightarrow \ \mathcal{E}_0 \cup \{0\}\,.
	\label{eq:ok4}
	\end{equation}
By Okada's second condition \eqref{eq:ok2}, the restriction \eqref{eq:ok4} is a square matrix. 
By the nondegeneracy condition $\det SR \neq 0$, the matrix $B$ possesses trivial kernel, and so does its restriction \eqref{eq:ok4}. Therefore the square restriction of $B$ is invertible, like the original matrix $B$ itself. 
In particular \eqref{eq:3.1} and \eqref{eq:ok4} imply
	\begin{equation}
	z^{j^*} := -B^{-1} e_{j^*} \in (\mathcal{E}_0 \cup \mathcal{M}_0),
	\label{eq:ok5}
	\end{equation}
for any $j^* \in \mathcal{E}_0$. By definition \eqref{eq:3.3} of influence, this proves Okada's upper estimate \eqref{eq:ok3}.

In spite of ambitious claims added by the second author in \cite{OkadaMochizuki16}, prematurely, Okada's elegant upper estimate does not establish any hierarchy of influence. 
The Okada estimate has only been shown to be compatible with the hierarchical structure already observed, by anecdotal evidence, in the examples of \cite{fiedler15example}. 
Our present paper provides a universal mathematical foundation for such hierarchy phenomena, for the first time, thanks to the multimolecular transitivity theorem \ref{thm:2.4}. 

The hierarchical structure which we obtain is necessarily more detailed than the upper estimates by Okada. 
We have seen how the Okada upper estimate may in fact strictly overestimate influence regions. 
But only when oversold as an influence hierarchy, the upper estimate \eqref{eq:ok3} oversimplifies the beautifully complex network response to single reaction perturbations and knockout experiments.

While working on the present paper we also came across the groundbreaking and monumental results by Murota on \emph{layered matrices}. 
Some of his results precede our present work by more than three decades.
See for example his monograph \cite{murota2009matrices} and the many references there.
Main applications, so far, have been concerned with electrical networks of Kirchhoff type.
Very abstractly, Murota's layered matrices contain entries from two different fields $\mathbb{K} \subset \mathbb{F}$. 
Our network analysis amounts to the particular case $\mathbb{K} = \mathbb{R}$ or $\mathbb{Q}$, and rational functions $\mathbb{F} = \mathbb{R}(x)$.
In his much more general setting, Murota derives normal forms, block triangularizations, and associated formal partial orders.
Block triangularization of our mixed matrix $B$, of course, is inherited by $B^{-1}$. 
``Hierarchic'' subspaces of influence can then be interpreted as a consequence of block triangularization of $B^{-1}$. 
Therefore Murota's formal partial orders are closely related to upper estimates of influence regions in the Okada style. 
In that sense, even our anecdotal observation of specific influence patterns in the examples of \cite{fiedler15example} cannot claim much mathematical novelty. 

Our notion of transitivity of influence by rate perturbations of single reactions $j^*$, on the other hand, is too closely related to the specific structure of metabolic networks to have been considered by Murota. Indeed, we did make use of the specific metabolite-reaction structure of chemical reaction systems, as expressed in the specific layered structure of the matrix $B$ and the subtle relation between the reaction matrix $R$ and the stoichiometric matrix $S$. Therefore it was easier for us, and probably for most of our readers, to develop the relevant theory from scratch and make our paper reasonably elementary and self-contained. 

In conclusion we hope that our theoretical attempts may serve the scientific community well enough, in face of the formidable challenges -- both, experimental and theoretical -- posed by the large and improving data bases of metabolic networks.



\begin{thebibliography}{999999999}

\small{
\bibitem[Fei95]{feinberg95}
Martin Feinberg.
``The existence and uniqueness of steady states for a class of chemical reaction networks''.
\emph{Arch. Rational Mech. Analysis} \textbf{132} (1995), 311--370.

\bibitem[Fel92]{fell92}
David A. Fell.
``Metabolic control analysis: a survey of its theoretical and experimental development''.
\emph{Biochem. J.} \textbf{286} (1992), 313--330.

\bibitem[FFPack]{fflas-ffpack}
The FFLAS-FFPACK group. 
\emph{FFLAS-FFPACK: Finite Field Linear Algebra Subroutines / Package}.
v2.0.0. \url{http://linalg.org/projects/fflas-ffpack} 2014.

\bibitem[FM15]{fiedler15mono}
Bernold Fiedler and Atsushi Mochizuki.
``Sensitivity of chemical reaction networks: a structural approach.
2. Regular monomolecular systems''.
\emph{Math. Meth. Appl. Sc.} \textbf{38} (2015), 3519--3537.

\bibitem[Gan77]{gantmacher}
Felix R. Gantmacher.
\emph{The Theory of Matrices}.
Vol. 1. AMS, Providence 1977.

\bibitem[Gio+15]{giordano2015computing}
Giulia Giordano, Christian Cuba Samaniego, Elisa Franco and Franco Blanchini.
``Computing the structural influence matrix for biological systems''.
\emph{J. Math. Biology} \textbf{72} (2016), 1927--1958.

\bibitem[GF06]{grofeu}
Thilo Gross and Ulrike Feudel.
``Generalized models as a universal approach to the analysis of nonlinear dynamical systems''.
\emph{Phys.~Rev.~E} \textbf{73} (2006), 016205.

\bibitem[HS96]{HeinrichSchuster}
Reinhart Heinrich and Stefan Schuster.
\emph{The Regulation of Cellular Systems}.
Chapman \& Hall, New York 1996.


\bibitem[Ish+07]{ishii} 
Nobuyoshi Ishii, Kenji Nakahigashi, Tomoya Baba, Martin Robert, Tomoyoshi Soga, Akio Kanai, Takashi Hirasawa, Miki Naba, Kenta Hirai, Aminul Hoque,  Masaru Tomita et al.
``Multiple high-throughput analyses monitor the response of E. coli to perturbations''.
\emph{Science} \textbf{316} (2007), 593--597.

\bibitem[KG00]{KEGG}
Minoru Kanehisa and Susumu Goto.
``KEGG: Kyoto encyclopedia of genes and genomes''.
\emph{Nucl. Acids Res.} \textbf{28} (2000), 27--30.

\bibitem[Le+06]{BernhardD}
Nicolas Le Novère, Benjamin Bornstein, Alexander Broicher, Mélanie Courtot, Marco Donizelli, Harish Dharuri, Lu Li, Herbert Sauro, Maria Schilstra, Bruce Shapiro, Jacky L. Snoep and Michael Hucka.
``BioModels Database: a free, centralized database of curated, published, quantitative kinetic models of biochemical and cellular systems''.
\emph{Nucl. Acids Res.} \textbf{34} (2006), D689--D691.

\bibitem[MV17]{MatanoVassena16}
Hiroshi Matano and Nicola Vassena.
``Monomolecular reaction networks: Flux-influenced sets and balloons.''
\emph{Math. Meth. Appl. Sci.} (2017), DOI: 10.1002/mma.4557.

\bibitem[MF15]{fiedler15example}
Atsushi Mochizuki and Bernold Fiedler.
``Sensitivity of chemical reaction networks: a structural approach.
1. Examples and the carbon metabolic network''.
\emph{J. Theor. Biology} \textbf{367} (2015), 189--202.

\bibitem[Mur09]{murota2009matrices}
Kazuo Murota.
\emph{Matrices and Matroids for Systems Analysis}.
Springer Science \& Business Media \textbf{20} 2009.

\bibitem[Nak+09]{ishii2}
Kenji Nakahigashi, Yoshihiro Toya, Nobuyoshi Ishii, Tomoyoshi Soga, Miki Hasegawa, Hisami Watanabe, Yuki Takai, Masayuki Honma, Hirotada Mori and Masaru Tomita.
``Systematic phenome analysis of Escherichia coli multiple-knockout mutants reveals hidden reactions in central carbon metabolism''.
\emph{Mol. Syst. Biology} \textbf{5} (2009), 306.

\bibitem[Oka+15]{OkadaFiedlerMochizuki15}
Takashi Okada, Atsushi Mochizuki and Bernold Fiedler.
``Perturbation-response relation and network topology
in biochemical reaction systems''.
Equadiff Lyon, Poster, 2015.

\bibitem[OM16]{OkadaMochizuki16}
Takashi Okada and Atsushi Mochizuki.
``Law of Localization in Chemical Reaction Networks''.
\emph{Phys. Rev. Letters} \textbf{117} (2016), 048101.

\bibitem[Pal06]{palsson05}
Bernhard Palsson.
\emph{Systems Biology. Properties of Reconstructed Networks}.
Cambridge Univ. Press 2006.

\bibitem[Sage]{sage} The Sage Developers. 
\emph{Sage Mathematics Software} (Version 7.0). \url{http://www.sagemath.org} 2016.

\bibitem[Sch80]{schwartz1980fast}
Jacob T. Schwartz.
``Fast probabilistic algorithms for verification of polynomial identities''.
\emph{J. ACM} \textbf{27} (1980), 701--717.

\bibitem[SF10]{sf10}
Guy Shinar and Martin Feinberg.
``Structural sources of robustness in biochemical reaction networks''.
\emph{Science} \textbf{327} (2010), 1389--1391.

\bibitem[SF11]{sf11}
Guy Shinar and Martin Feinberg.
``Design principles for robust biochemical reaction networks: what works, what cannot work, and what might almost work''.
\emph{Math. Biosci.} \textbf{231} (2011), 39--48.

\bibitem[SF13]{shinarfeinberg}
Guy Shinar and Martin Feinberg.
``Concordant chemical reaction networks and the species-reaction graph''.
\emph{Math. Biosc.} \textbf{241} (2013), 1-23.

\bibitem[Shi+09]{sf09}
Guy Shinar, Uri Alon and Martin Feinberg.
``Sensitivity and robustness in chemical reaction networks''.
\emph{SIAM J. Appl. Math.} \textbf{69} (2009), 977--998.

\bibitem[Shi+11]{shietal11}
Guy Shinar, Avi Mayo, Haixia Ji and Martin Feinberg.
``Constraints on Reciprocal Flux Sensitivities in Biochemical Reaction
Networks''.
\emph{Biophys. J.} \textbf{100} (2011), 1383--1391.

\bibitem[Ste84]{stephanopoulos}
George Stephanopoulos.
\emph{Chemical Process Control: An Introduction to Theory and Practice}.
Prentice-Hall 1984.

\bibitem[Steu+07]{groTCA}
Ralf Steuer, Adriano Nunes Nesi, Alisdair R. Fernie, Thilo Gross, Bernd Blasius, Joachim Selbig.
``From structure to dynamics in metabolic pathways: application to the plant mitochondrial TCA cycle''.
\emph{Bioinformatics} \textbf{23} (2007), 1378--1385.

\bibitem[Vas15]{Vassena15}
Nicola Vassena.
``Monomolecular reaction networks: a new proof of flux transitivity''.
Equadiff Lyon, Poster, 2015.

\bibitem[Vas16]{Vassena16}
Nicola Vassena.
\emph{Monomolecular reaction networks: a new proof of flux transitivity}. Master Thesis, Free University of Berlin, 2016.
}



\end{thebibliography}
\end{document}